\documentclass[12pt]{article}

\usepackage{amsmath}
\usepackage{amssymb}
\usepackage{fullpage}
\usepackage{amsthm}
\usepackage{graphicx}
\usepackage{placeins}
\usepackage{extpfeil}
\usepackage{enumerate}
\usepackage{caption}
\usepackage{subcaption}
\usepackage{hyperref}
\usepackage{xcolor}
\usepackage{pbox}
\usepackage[numbers,sort&compress]{natbib}

\hypersetup{hidelinks}

\newtheorem{thm}{Theorem}[section]
\newtheorem{theorem}{Theorem}
\newtheorem{prop}{Proposition}[section]

\newtheorem{rmk}{Remark}[section]
\newtheorem{lem}{Lemma}[section]

\newcommand{\dive}{{\rm div \hspace{0.05cm} }}

\newcommand{\R}{\mathbb{R}}
\newcommand{\bS}{\mathbb{S}}

\newcommand{\supp}{\mathrm{supp}\hspace{0.05cm}}

\newcommand{\bmx}{\begin{bmatrix}}
\newcommand{\emx}{\end{bmatrix}}
\newcommand{\be}{\begin{equation}}
\newcommand{\ee}{\end{equation}}

\begin{document}

\title{Classification of local $(-1)$-homogeneous axisymmetric solutions of the $3$D stationary Navier-Stokes equations with logarithmic singular-ray behavior}
\author{Xukai Yan\footnote{Department of Mathematics, Oklahoma State University, 401 Mathematical Sciences Building, Stillwater, OK 74078, USA. Email: xuyan@okstate.edu. Partially supported by NSF Career Award DMS-2441137 and Simons Foundation Travel Support for Mathematicians 962527.}}
\date{}
\maketitle

\abstract{We study the existence and classification of local $(-1)$-homogeneous axisymmetric solutions $(u, p)$ of the three-dimensional incompressible stationary Navier-Stokes equations near a singular ray. We focus on such solutions with Type II behavior, which satisfy $0<\limsup_{x\in\bS^2, x\to P}|u|/|\ln \text{dist}(x, P)|<\infty$, where $P$ is the south or north pole. 
 %satisfy  $|u|=O(|\ln \text{dist}(x, P)|)$ as $x\to P$ on $\bS^2$, where $P$ is the south or north pole. 
  Apart from the Landau solutions, these are the least singular  nontrivial $(-1)$-homogeneous axisymmetric solutions.
  We construct a four-parameter family of local $(-1)$-homogeneous axisymmetric solutions with at most logarithmic growth near the singular ray and represent them as convergent series whose coefficients are determined recursively. 
  Conversely, we prove that every local  $(-1)$-homogeneous axisymmetric solution  satisfying $|u|=o(\text{dist} (x, P)^{-1})$ as $x\to P$ on $\bS^2$ belongs to this family. 
%We  construct a four-parameter family of such solutions and represent them as convergent series whose coefficients are determined recursively. Conversely, we prove that every local  nontrivial $(-1)$-homogeneous axisymmetric solution  satisfying $|u|=o(\text{dist} (x, P)^{-1})$ as $x\to P$ on $\bS^2$ belongs to this family if it is not a Landau solution.
In particular, this family %the constructed family of solutions
 contains Type II solutions with nonzero swirl, in contrast to the global $(-1)$-homogeneous axisymmetric Type II solutions in $\R^3\setminus\{x'=0\}$, which have no swirl.  We also establish derivative estimates for the constructed solutions and identify the singular force generated by these solutions across the singular ray in the sense of distributions. }

\section{Introduction}\label{sec_1}

Consider the incompressible stationary Navier-Stokes equations in $\mathbb{R}^3$,
\begin{equation}\label{NS}
	\left\{
	\begin{aligned}
    	& -\Delta u + ( u \cdot \nabla ) u + \nabla p = 0, \\
		& \dive u=0,
   \end{aligned}
   \right.
\end{equation}
where $u: \mathbb{R}^3\to\mathbb{R}^3$ is the velocity field and $p:\mathbb{R}^3\to\mathbb{R}$ is the pressure. The system is invariant under the scaling $u(x)\to \lambda u(\lambda x)$ and $p(x)\to \lambda^2 p(\lambda x)$ for any $\lambda>0$. It is therefore natural to study solutions invariant under this scaling,
for which $u$ is $(-1)$-homogeneous and $p$ is $(-2)$-homogeneous. Throughout this paper we refer to them as $(-1)$-homogeneous solutions.
More generally, a function $f$ is said to be $(-\alpha)$-homogeneous if $f(\lambda x)= \lambda^{-\alpha} f(x)$ for every $\lambda>0$.

Write $x=(x_1,x_2,x_3)\in \R^3$ and $x'=(x_1, x_2)$. Let $(r, \phi, x_3)$ be the standard cylindrical coordinates, where $r=|x'|$ and $\phi$ is the azimuthal angle about the $x_3$-axis.
For $r>0$, a vector field $u$ can be written as $u = u^r e_r +  u^\phi e_{\phi}+u^3e_3$, where $e_r,  e_{\phi}, e_3$ are the associated orthonormal basis vectors. We say a solution $(u, p)$ to (\ref{NS}) is \emph{axisymmetric} if $u^r$,  $u^{\phi}, u^3$ and $p$ are independent of $\phi$, and \emph{no-swirl} if $u^{\phi}\equiv 0$.

In 1944, Landau \cite{Landau} discovered a three-parameter family of explicit $(-1)$-homogeneous solutions of the stationary Navier-Stokes equations in $C^\infty(\mathbb{R}^3\setminus\{0\})$, see also \cite{SL, SQ}.
These solutions, now known as \emph{Landau solutions}, are axisymmetric, no-swirl, and have one singular point at the origin.
Tian and Xin proved in \cite{TianXin} that all nontrivial $(-1)$-homogeneous, axisymmetric solutions of (\ref{NS}) in $C^\infty(\mathbb{R}^3\setminus\{0\})$ are Landau solutions.  \v{S}ver\'{a}k \cite{Sverak} later removed the axisymmetry assumption and established the following result:

\renewcommand{\thetheorem}{\Alph{theorem}}
\begin{theorem}[\cite{Sverak}]\label{thm:Sverak}
	All nonzero $(-1)$-homogeneous  solutions of \eqref{NS} in $C^2(\mathbb{R}^3\setminus\{0\})$ are Landau solutions.
\end{theorem}
\renewcommand{\thetheorem}{theorem}

  For further results on homogeneous solutions of (\ref{NS}), see, for example,   \cite{CKPW,  G, Gu, JLL, KT, LZZ, PP1, PP2, PP3, Serrin, SL, SQ, W, Y, ZZ}.

Since $(-1)$-homogeneous solutions  of (\ref{NS}) are determined by their traces on $\bS^2$, the system (\ref{NS}) may be reduced to a system on $\bS^2$.
Theorem \ref{thm:Sverak} gives a complete classification of $(-1)$-homogeneous solutions of  (\ref{NS}) whose traces are in $C^2(\mathbb{S}^2)$.
A natural next problem is to study $(-1)$-homogeneous solutions of (\ref{NS}) whose traces have finitely many isolated singularities on $\bS^2$. By homogeneity, a  singular point $P\in \bS^2$ corresponds to a singular ray in $\R^3$ emanating from the origin and passing through $P$.
Hence, $(-1)$-homogeneous solutions of (\ref{NS}) with finitely many singular rays in $\mathbb R^3$ are those whose traces have finitely many isolated singularities on $\mathbb S^2$.

Motivated by this problem, in a series of works  \cite{LLY1, LLY2, LLY3, LLY4, LLY5, LY}, the author and collaborators investigated $(-1)$-homogeneous  solutions of (\ref{NS}) in $C^2(\mathbb{S}^2\setminus\{S, N\})$, where $S$ is the south pole and $N$ is the north pole. These works established classifications of axisymmetric no-swirl solutions \cite{LLY1, LLY2}, existence and nonexistence results for axisymmetric solutions with nonzero swirl \cite{LLY1, LLY3},
results on the vanishing viscosity limits of axisymmetric no-swirl solutions \cite{LLY4}, an optimal removable singularity theorem \cite{LLY5}, and asymptotic stability of small axisymmetric no-swirl solutions with logarithmic singular ray behavior \cite{LY}.

The asymptotic expansions obtained in \cite{LLY1}  for $(-1)$-homogeneous  axisymmetric solutions of (\ref{NS})  suggest the weakest possible asymptotic behavior of general $(-1)$-homogeneous solutions of (\ref{NS}) near a nonremovable singular ray is of logarithmic order, in the sense that $|u|=O(|\ln |x'||+1)$ as $|x'|\to 0$ for each $x_3\ne 0$.
The following removable singularity result in \cite{LLY5} confirms that  logarithmic growth is the critical threshold without assuming axisymmetry.

\renewcommand{\thetheorem}{\Alph{theorem}}
\begin{theorem}[\cite{LLY5}]\label{thm_B}
Let $P\in\mathbb S^2$, and let $(u,p)$ be a local $C^2$ $(-1)$-homogeneous solution of (\ref{NS}) in a punctured neighborhood of $P$ on $\mathbb S^2$.
	If $\lim_{x\in \mathbb{S}^2, x\to P}|u(x)|/|\ln \textrm{dist} (x, P)|=0$,
	then $(u, p)$ can be extended smoothly across $P$ on $\bS^2$.
\end{theorem}
\renewcommand{\thetheorem}{theorem}

Theorem \ref{thm_B} is optimal: for any $\alpha>0$, there exist $(-1)$-homogeneous solutions in $C^{\infty}(\mathbb{S}^2\setminus\{S, N\})$ with nonremovable singularities at $S$ and $N$ such that \\
$\lim_{x\in \mathbb{S}^2, x\to P}|u(x)|/|\ln \textrm{dist} (x, P)|=\alpha$, where $P=S$ or $N$, see \cite{LLY2} and \cite{LLY5} for examples of such solutions.

  The asymptotic expansions obtained in \cite{LLY1, LLY2, LLY3} for $(-1)$-homogeneous axisymmetric solutions of (\ref{NS}) in $C^{2}(\mathbb{S}^2\setminus\{S, N\})$  show that such solutions  fall into three mutually exclusive classes according to their behavior near the singular $x_3$-axis:
    \begin{enumerate}
	\item[]Type I. $\sup_{|x|=1}|u(x)|<\infty$;
	\item[]Type II.  $0<\limsup_{|x|=1,x'\to 0}|u(x)|/| \ln |x'||<\infty$;
	\item[]Type III.  $\limsup_{|x|=1,x'\to 0}|x'||u(x)|>0$.
\end{enumerate}
By Theorem \ref{thm_B}, every $(-1)$-homogeneous  solution of (\ref{NS}) in $C^{2}(\mathbb{S}^2\setminus\{S, N\})$ with Type I behavior extends smoothly across $S$ and $N$ on $\bS^2$.
  Then by Theorem \ref{thm:Sverak}, such nonzero solutions are precisely the Landau solutions.

Type II solutions exhibit logarithmic growth near the singular axis and therefore lie precisely at the threshold of the removable singularity criterion in Theorem \ref{thm_B}, while Type III solutions have stronger algebraic singularity behavior. One natural  question is to study the existence and classification of Type II $(-1)$-homogeneous solutions of (\ref{NS}).
 It was proved in \cite{LLY5} that every Type II $(-1)$-homogeneous axisymmetric solution in $C^{\infty}(\mathbb{S}^2\setminus\{S, N\})$ must be no-swirl and singular at both $S$ and $N$. In particular, there is no $(-1)$-homogeneous axisymmetric Type II solution in $C^2(\mathbb S^2\setminus\{P\})$ with only one singularity at $P=S$ or $N$. These global $(-1)$-homogeneous axisymmetric no-swirl Type II solutions were classified in \cite{LLY2}, although they do not have explicit expressions.

 %The preceding nonexistence result for global Type II $(-1)$-homogeneous axisymmetric  solutions of (\ref{NS}) with nonzero swirl does not rule out the existence of local  Type II $(-1)$-homogeneous axisymmetric solutions of (\ref{NS}) with nonzero swirl near a single ray. 
 The preceding nonexistence result for global Type II $(-1)$-homogeneous axisymmetric  solutions of (\ref{NS}) with nonzero swirl does not rule out the existence of such local  solutions near a single ray. 
 In this paper, we study the existence, classification, and asymptotic behavior of local $(-1)$-homogeneous axisymmetric  solutions of (\ref{NS}) that exhibit at most Type II singular behavior near a ray. %near a singular ray.  %local $(-1)$-homogeneous axisymmetric  solutions of (\ref{NS}) near a singular ray with at most Type II singular behavior, %such local solutions of (\ref{NS})  near a singular ray,  focusing on their existence, classification, and asymptotic behavior. 
 By rotational invariance, it suffices to study such solutions near the north pole on $\bS^2$, i.e., solutions near the positive $x_3$-axis in $\R^3$.  We  construct a four-parameter family of  such solutions and represent them as convergent series whose coefficients are determined recursively. We then prove that every  local   $(-1)$-homogeneous axisymmetric solution of (\ref{NS}) satisfying $|u|=o(|x'|^{-1})$ as $x'\to 0$ for each fixed $x_3>0$
 belongs to this family. %the family of solutions we construct. % if it is not a Landau solution.  
  In particular, this family   contains Type II solutions with nonzero swirl, in contrast to the global Type II $(-1)$-homogeneous axisymmetric solutions studied in  \cite{LLY2, LLY5}, which necessarily have no swirl.

 We also identify the singular force generated by these solutions across the singular ray in the sense of distributions. This provides an analogue of the point force generated by Landau solutions, with our solutions generating a line force along   the singular ray.

 %To formulate the local problem near $N$, 
 We work in a cone centered at the positive $x_3$-axis.
 Throughout this paper, we assume $x_3>0$.  Recall that $x'=(x_1, x_2)$ and $r=|x'|$, and let $\Delta'=\partial_1^2+\partial_2^2$. Introduce the scale-invariant variables
\[
   \rho:=\frac{r}{x_3}, \quad t=\ln\rho.
   \]
   We retain the cylindrical orthonormal frame and write $e_{\rho}:=e_{r}$. Accordingly,
$u=u^{\rho}e_{\rho}+u^{\phi}e_{\phi}+u^3e_3$ and $u'=u^1e_1+u^2e_2=u^{\rho}e_{\rho}+u^{\phi}e_{\phi}$. For  $\delta>0$, define
\be\label{eq_Cone}
\Omega_{\delta}:=\{x\in \R^3\mid |x'|<\delta x_3, x_3>0\}=\{x\in \R^3\mid 0\le \rho< \delta, x_3>0\}.
\ee

We will show that, after possibly decreasing $\delta$,
every $(-1)$-homogeneous axisymmetric solution $(u, p)$  in $\Omega_{\delta}\setminus\{x'=0\}$ satisfying  $|u|=o(|x'|^{-1})$ as $x'\to 0$ admits the expansions
\begin{equation}\label{eq1_1}
    u= \sum_{n=0}^{\infty}u_n, \quad p= \sum_{n=0}^{\infty}p_n,\quad \textrm{ in }\Omega_{\delta}\setminus\{x'=0\}.
  \end{equation}
  Both series converge in $C^2(V)$ for every compact set  $V\subset\subset\Omega_{\delta}\setminus\{x'=0\}$.  The terms $\{(u_n, p_n)\}_{n=0}^{\infty}$ satisfy the following structural properties:

  \medskip

	\noindent\textbf{(A1)} \emph{Structure of expansion terms: Each $u_n$ is $(-1)$-homogeneous and each $p_n$ is $(-2)$-homogeneous. Moreover, $u_n=u_n^{\rho}e_{\rho}+u_n^{\phi}e_{\phi}+u_n^3e_3$, where for each $n\ge 0$,
	  \begin{equation*}
	  u^{\rho}_n=x_3^{-1}\rho^{2n+1}a^{\rho}_{n}(t), \ u^{\phi}_n=x_3^{-1}\rho^{2n+1}a^{\phi}_{n}(t), \ u^{3}_n=x_3^{-1}\rho^{2n}a^3_{n}(t), \ p_n=x_3^{-2}\rho^{2n}c_{n}(t),
  \end{equation*}
 in $ \Omega_{\delta}\setminus\{x'=0\}$,  where  $ c_{n}(t)$ and $a^{j}_{n}(t)$, $j=\rho, \phi, 3$, are polynomials of $t=\ln \rho$ satisfying
  \be\label{eq1_A2}
    \deg a_{n}^{j}\le n+1, \quad \deg c_n\le n+1.
  \ee
  }

\noindent \textbf{(A2)} \emph{Leading-order term:  $(u_0, p_0)$ is given by
  \begin{equation}\label{eq1_a_2}
  u_0^{\rho}=x_3^{-1}\rho(\frac{1}{2}c_0t+b_{0}^{\rho}), \quad u_0^{\phi}=b_{0}^{\phi}x_3^{-1}\rho, \quad u_0^3=x_3^{-1}(c_0t+b_{0}^3), \quad p_0=x_3^{-2}(c_0t+d_0),
  \end{equation}
  where $c_0, d_0, b_{0}^{\rho}, b_{0}^{\phi}$ are constants uniquely determined by $(u, p)$ through (\ref{eq1_a_3}) below,
  and $b_{0}^3=2b_{0}^{\rho}-\frac{1}{2}c_0$.
}

\medskip

\noindent \textbf{(A3)} \emph{Recursive relation:  For every $n\ge 0$, $(u_{n+1}, p_{n+1})$ is uniquely determined, within the class specified in \textbf{(A1)}, by  the lower-order terms through
   \begin{equation}\label{eq1_3}
   \left\{
    \begin{split}
        & -\Delta' p_{n+1}=\partial_3^2p_n+\sum_{m=0}^n\dive (u_m\cdot \nabla u_{n-m})=:f_n^p\\
         & \Delta' u'_{n+1}=-\partial_3^2u_n'+\sum_{m=0}^nu_m\cdot \nabla u_{n-m}'+\nabla_{x'} p_{n+1}=:f_n',\\
         & \Delta' u^3_{n+1}=-\partial^2_3u^3_n+\sum_{m=0}^nu_m\cdot \nabla u_{n-m}^3+\partial_3p_n=:f_n^3.
       \end{split}
    \right.
  \end{equation}
  }

  \medskip

\medskip

 \noindent The expansion (\ref{eq1_1}) may be viewed as a Frobenius expansion for the nonlinear system (\ref{NS}) near the singular ray along the positive  $x_3$-axis. The four constants $c_0, d_0, b_{0}^{\rho}, b_{0}^{\phi}$ in \textbf{(A2)} are uniquely determined by $(u, p)$ on $\{x_3=1\}$ through
  \begin{equation}\label{eq1_a_3}
   \begin{split}
   & c_0=\lim_{\rho\to 0^+}\rho\partial_{\rho}p, \ d_0=\lim_{\rho\to 0^+}(p-c_0\ln\rho),\  b_{0}^{\rho}= \lim_{\rho\to 0^+}\frac{1}{\rho}(u^{\rho}-\frac{1}{2}c_0\rho \ln\rho), \ b_{0}^{\phi}= \lim_{\rho\to 0^+}\frac{1}{\rho}u^{\phi}.
     \end{split}
  \end{equation}
 Denote $\lambda:=(c_0, d_0, b_0^{\rho}, b_0^{\phi})\in \R^4$.   The sequence $\{(u_n, p_n)\}_{n=0}^{\infty}$ is uniquely determined by $\lambda$ through \textbf{(A1)}--\textbf{(A3)}, and satisfies $\dive u_n=0$ for all $n\ge 0$.

\begin{thm}\label{thm_main}
  Let $K\subset \R^4$ be compact.    There exists $\delta>0$, depending only on $K$, such that for every $\lambda=(c_0, d_0, b_{0}^{\rho}, b_{0}^{\phi})\in K$, there exists a unique $(-1)$-homogeneous axisymmetric solution $(u_{\lambda}, p_{\lambda})\in C^{\infty}(\Omega_{\delta}\setminus\{x'=0\})$ of (\ref{NS})   such that %satisfying
   $|u_{\lambda}|=o(|x'|^{-1})$ as $x'\to 0$ for each fixed $x_3>0$ and
 (\ref{eq1_a_3}) holds with prescribed parameter $\lambda$. Moreover,
   \[
    u_{\lambda}=\sum_{n=0}^{\infty}u_n, \quad p_{\lambda}=\sum_{n=0}^{\infty}p_n,
  \]
   where $\{(u_n, p_n)\}_{n=0}^{\infty}$ is uniquely determined by $\lambda$ through \textbf{\textup{(A1)}}--\textbf{\textup{(A3)}}. For every compact set $V\subset\subset\Omega_{\delta}\setminus\{x'=0\}$, both series converge in $C^2(V)$ uniformly for $\lambda\in K$.

   Conversely, let $\tilde{\delta}>0$ and $(u, p)$
    be a $C^2$ $(-1)$-homogeneous axisymmetric solution of  (\ref{NS}) in $\Omega_{\tilde{\delta}}\setminus\{x'=0\}$ satisfying $|u|=o(|x'|^{-1})$ as $x'\to 0$ for each fixed $x_3>0$. Then the limits in (\ref{eq1_a_3}) exist and determine a unique $\lambda=(c_0, d_0, b_{0}^{\rho}, b_{0}^{\phi})\in\R^4$, and  there exists $0<\delta_1\le \tilde{\delta}$ such that  $(u, p)=(u_{\lambda}, p_{\lambda})$ in $\Omega_{\delta_1}\setminus\{x'=0\}$, where $(u_{\lambda}, p_{\lambda})$ is the solution given by the first part with $K=\{\lambda\}$.
      \end{thm}

%It follows from (\ref{eq1_a_2}) that 
The solution $(u_{\lambda}, p_{\lambda})$ constructed in Theorem \ref{thm_main} is of Type II if and only if $c_0\ne 0$, and has nonzero swirl if and only if $b_0^{\phi}\ne 0$. Theorem \ref{thm_main}  provides a complete classification of all local Type II $(-1)$-homogeneous axisymmetric solutions of (\ref{NS}) near the singular ray $\{x'=0, x_3>0\}$. To the best of our knowledge, this is the first complete classification  of Type II solutions of (\ref{NS}) within a class that allows nonzero swirl.

To prove Theorem \ref{thm_main}, we reformulate the system (\ref{NS}) for $(-1)$-homogeneous axisymmetric solutions on $\{x_3=1, x'\ne 0\}$ in terms of the scale-invariant variable $\rho$.  
Starting from the leading logarithmic terms, we recursively construct a formal series solution with the structure described in \textbf{\textup{(A1)}}--\textbf{\textup{(A3)}}. %whose general terms involve products of powers of $\rho^2$ and  $\ln\rho$. 
The key step is to prove the convergence of this series, for which we establish quantitative  derivative estimates for its individual terms  using carefully chosen majorants. 
%The existence of solutions is proved by recursively constructing Frobenius-type series described by \textbf{(A1)}-\textbf{(A3)} in powers of $\rho^2$, whose coefficients are polynomials in $\ln\rho$. 
%The convergence of the series solutions is established  suitable majorants. 
Conversely, for any  $(-1)$-homogeneous axisymmetric solution $(u,p)$ near the positive $x_3$-axis %in $\Omega_{\delta}\setminus\{x'=0\}$ for some $\delta>0$ 
satisfying $|u|=o(|x'|^{-1})$, we successively identify its expansion terms and derive precise estimates for the remainders, showing that $(u, p)$ coincides with one of the constructed solutions near the singular ray.   This yields the uniqueness and classification of such solutions. %The uniqueness of solutions follows from estimates of remainder terms in the expansion of any $(-1)$-homogeneity axisymmetry solution. 

We also establish the derivative estimates for the solutions obtained in Theorem \ref{thm_main}.
Let $\mathbb{N}$ denote the set of nonnegative integers.
 For $k, l\in\mathbb{N}$ and a multi-index $\beta\in\mathbb{N}^4$, let $\nabla_{x'}^k\partial_3^l\partial_{\lambda}^{\beta}f$  denote a mixed derivative of order $k$ in $x'$, order $l$ in $x_3$, and order $|\beta|$ in the parameter $\lambda$.

  \begin{prop}\label{prop1_1}
   Let $K\subset \R^4$ be compact and let $\delta=\delta(K)>0$ be as in Theorem \ref{thm_main}.   For each $\lambda=(c_0, d_0, b_{0}^{\rho}, b_{0}^{\phi})\in K$, let $(u_{\lambda}, p_{\lambda})$ be the corresponding $(-1)$-homogeneous axisymmetric solution of (\ref{NS}) in $\Omega_{\delta}\setminus\{x'=0\}$ given by Theorem \ref{thm_main}. Then $(u_{\lambda}, p_{\lambda})$ depends smoothly on $\lambda$.    Moreover,  for every integer $L\ge 0$, there exists a constant $C>0$, depending only on $K$ and $L$,
   such that for every multi-index $\beta\in \mathbb{N}^4$ and  $k, l\in \mathbb{N}$ satisfying $|\beta|+k+l\le L$, and every $x\in \Omega_{\delta}\setminus\{x'=0\}$,
      \be\label{eqthm1_2}
       \begin{split}
       &  |\nabla_{x'}^k\partial_3^l\partial_{\lambda}^{\beta}u'_{\lambda}|
         \le \left\{
         \begin{array}{ll}
          Cx_3^{-k-l-1}\rho^{-k+1}(|\ln\rho|+1), &  \textrm{if }k=0\textrm{ or } k=1,\\
          Cx_3^{-k-l-1}\rho^{-k+1}, & \textrm{if }k\ge 2.
         \end{array}
         \right.\\
       &   |\nabla_{x'}^k\partial_3^l\partial_{\lambda}^{\beta}u_{\lambda}^3|
         \le \left\{
         \begin{array}{ll}
          Cx_3^{-k-l-1}(|\ln\rho|+1), & \textrm{if }k=0, \\
          Cx_3^{-k-l-1}\rho^{-k}, & \textrm{if }k\ge 1.
         \end{array}
         \right.\\
       &  |\nabla_{x'}^k\partial_3^l\partial_{\lambda}^{\beta}p_{\lambda}|
         \le \left\{
         \begin{array}{ll}
          Cx_3^{-k-l-2}(|\ln\rho|+1), &  \textrm{if }k=0, \\
          Cx_3^{-k-l-2}\rho^{-k}, &  \textrm{if }k\ge 1.
         \end{array}
         \right.
         \end{split}
              \ee\end{prop}

The following proposition identifies the singular force generated by the local Type II solutions of (\ref{NS}) constructed in Theorem \ref{thm_main} in the sense of distributions. This describes how these solutions satisfy (\ref{NS})  across the singular ray.
Landau \cite{Landau} interpreted his solutions as jets discharged from the origin and driven by a force represented by a Dirac mass (see, e.g., \cite{CK04}).
In contrast, solutions with singular rays, such as those constructed in \cite{LLY1, LLY2, LLY3} and in the present paper, generate forces supported along those rays.
In \cite{LY}, it was shown that global Type II $(-1)$-homogeneous axisymmetric solutions of (\ref{NS}) on $\bS^2\setminus\{S, N\}$ generate a point force at the origin together with a line force supported on the $x_3$-axis. Although the local Type II solutions constructed here may have nonzero swirl, their leading singular behavior near the singular ray agrees with that of the global Type II solutions, and they generate the same type of  line forces.

\begin{prop}\label{propF}
  Let $\lambda=(c_0, d_0, b_0^{\rho}, b_0^{\phi})\in \R^4$, and  let $(u_{\lambda}, p_{\lambda})$ be the  $(-1)$-homogeneous axisymmetric solution of (\ref{NS}) given by Theorem \ref{thm_main} in $\Omega_{\delta}\setminus\{x'=0\}$ for some $\delta>0$.
  Then,
  \be\label{eqF_1}
     \left\{
       \begin{split}
         & -\Delta u_{\lambda}+u_{\lambda}\cdot \nabla u_{\lambda}+\nabla p_{\lambda}=-2\pi c_0\partial_{x_3}(\ln |x_3|\delta_{x'=0})e_3, \\
         & \dive u_{\lambda}=0,
       \end{split}
     \right.
  \ee
  in the sense of distributions in $\Omega_{\delta}$.
\end{prop}
The parameter $c_0$ %, which is the coefficient of the leading logarithmic terms of $(u_{\lambda}, p_{\lambda})$, also 
determines the strength of the singular line force.  In Proposition \ref{propF}, $\delta_{x'=0}$ denotes the Dirac measure in the $x'$-variable, supported on the positive $x_3$-axis. More precisely, (\ref{eqF_1}) means that for
 any $\varphi\in C_c^{\infty}(\Omega_{\delta})$ and $j=1, 2, 3$,
    \begin{equation}\label{eqF_2}
           \int_{\Omega_{\delta}} (\nabla u^j_{\lambda} \cdot \nabla \varphi-u^i_{\lambda} u^j_{\lambda} \partial_i\varphi-p_{\lambda}\partial_j\varphi)dx= 2\pi c_0\delta_{j3}\int_{0}^{\infty}\ln|x_3|\partial_3\varphi(0,0,x_3)dx_3,          \end{equation}
         and
         \begin{equation*}
            \int_{\Omega_{\delta}}u_{\lambda} \cdot\nabla \varphi dx=0.
         \end{equation*}
Note that only the line force appears here, and no point force arises since the test function satisfies $\supp \varphi\subset\subset \Omega_{\delta}$.
The local Type II solutions may be viewed as stationary flows driven by a singular line force concentrated along the positive $x_3$-axis.

\begin{rmk}
  We conclude with a remark on the asymptotic stability of the Type II solutions constructed in Theorem \ref{thm_main}. Karch and Pilarczyk \cite{Karch} proved the asymptotic stability of small Landau solutions under $L^2$-perturbations. In \cite{LY}, the asymptotic stability of small global Type II $(-1)$-homogeneous axisymmetric no-swirl solutions of (\ref{NS}) was established under $L^2$-perturbations.
   In view of Proposition \ref{prop1_1} and Proposition \ref{propF},
  the arguments in \cite{Karch,LY} can be adapted to show that
  small local Type II solutions $(u_{\lambda}, p_{\lambda})$ are also asymptotically stable under $L^2$-perturbations.
  This conclusion also follows from the recent asymptotic stability result  in \cite{BD26} for a broader class of small solutions to (\ref{NS}) on general domains. We omit the details here.
\end{rmk}

The paper is organized as follows. In Section \ref{sec_2}, we reformulate the system (\ref{NS}) for $(-1)$-homogeneous axisymmetric solutions on $\{x_3=1\}$. In Section \ref{sec_3}, we construct and classify local  $(-1)$-homogeneous axisymmetric solutions of (\ref{NS}) near the positive $x_3$-axis satisfying $|u|=o(|x'|^{-1})$ as $x'\to 0$ for each fixed $x_3>0$, and prove Theorem \ref{thm_main} and Proposition \ref{prop1_1}. In Section \ref{sec_F}, we compute the singular force generated by these solutions and prove Proposition \ref{propF}.

\section{Preliminaries}\label{sec_2}

Let $\delta>0$ be fixed. We study $(-1)$-homogeneous axisymmetric  solutions $(u, p)$ of (\ref{NS}) in the cone $\Omega_{\delta}\setminus\{x'=0\}$ defined by (\ref{eq_Cone}). We first reformulate the system (\ref{NS}) for such solutions. By homogeneity, for $x\in \Omega_{\delta}\setminus\{x'=0\}$, we have
\[
  u(x)=\frac{1}{x_3}u(\frac{x'}{x_3}, 1), \quad p(x)=\frac{1}{x_3^2}p(\frac{x'}{x_3}, 1).
  \]
Hence it suffices to study the system (\ref{NS}) on the cross section $\{0<|x'|<\delta, x_3=1\}$. Recall that
\[
  r=|x'|, \quad \rho=\frac{r}{x_3}, \quad t=\ln \rho.
\]
Let $\alpha\in\R$  and let $f$ be a $(-\alpha)$-homogeneous scalar function. Then
\be\label{eqE_0_0}
   f(x)=\frac{1}{x_3^{\alpha}}f(\frac{x'}{x_3}, 1)=\frac{1}{x_3^{\alpha}}f(\rho, \phi), \quad x_3>0.
\ee
Consequently,
\begin{equation}\label{eqE_0_1}
\begin{split}
    &   \partial_r f(x)=\frac{1}{x_3^{{\alpha}+1}}\partial_{\rho}f(\rho, \phi), \quad
   \partial_3f(x)    =-\frac{1}{x_3^{{\alpha}+1}\rho^{{\alpha}-1}}\partial_{\rho}(\rho^{\alpha}f), \quad
  \partial_3^2f(x)=\frac{1}{x_3^{{\alpha}+2}\rho^{\alpha}}\partial_{\rho}(\rho^2\partial_{\rho}(\rho^{\alpha}f)).
   \end{split}
\end{equation}
If $f$ is axisymmetric, then
\begin{equation}\label{eqE_0_2}
  \nabla f=\frac{1}{x_3^{{\alpha}+1}}(\partial_{\rho}fe_{\rho}-\frac{1}{\rho^{{\alpha}-1}}\partial_{\rho}(\rho^{\alpha} f)e_3),
\end{equation}
\begin{equation}\label{eqE_0_3}
  \Delta f=\frac{1}{x_3^{{\alpha}+2}}(\frac{1}{\rho}\partial_{\rho}(\rho\partial_{\rho}f)+\frac{1}{\rho^{\alpha}}\partial_{\rho}(\rho^2\partial_{\rho}(\rho^{\alpha} f))), \quad  \Delta' f=\frac{1}{x_3^{{\alpha}+2}\rho}\partial_{\rho}(\rho\partial_{\rho}f).  \end{equation}
By a slight abuse of terminology, throughout this paper, a $(-\alpha)$-homogeneous function on $D\times\{x_3=1\}$ for a set $D\subset \R^2$ will be identified with its  homogeneous extension to the cone $\{x\in \mathbb{R}^3\mid x'/x_3\in D, x_3>0\}$ given by (\ref{eqE_0_0}).

By (\ref{eqE_0_1}) and the relation between $\nabla_{x'}$ and $\partial_{\rho}$ for axisymmetric functions and vector fields, where for vector fields the derivatives of the frame vectors
$e_{\rho}, e_{\phi}$ are included, we have the following estimate: for any $(-\alpha)$-homogeneous axisymmetric function or vector field $f$, and any $k, l\in\mathbb{N}$,  there exists a constant $C(k, l, \alpha)>0$, such that
\be\label{eqE_0_4}
|\nabla_{x'}^k\partial_3^lf(\rho, 1)|\le C(k, l, \alpha)\sum_{i=0}^{k+l}\rho^{-k+i}|\partial_{\rho}^if(\rho, 1)|, \quad \rho>0.
\ee

Let $u, v$ be $(-1)$-homogeneous axisymmetric vector fields.
 On $\{x_3=1\}$, the scalar Laplacian of each cylindrical component $u^j$, $j=\rho, \phi, 3$ is given by
\begin{equation}\label{eqE_1}
  \Delta u^j=\frac{1}{\rho}\partial_{\rho}(\rho\partial_{\rho}u^j)+\frac{1}{\rho}\partial_{\rho}(\rho^2\partial_{\rho}(\rho u^j)).
\end{equation}
We also have, on $\{x_3=1\}$, that
\begin{equation}\label{eqE_6}
\begin{split}
 u\cdot \nabla v & =(u^{\rho}\partial_{\rho}v^{\rho}-u^3\partial_{\rho}(\rho v^{\rho})-\frac{u^{\phi}v^{\phi}}{\rho})e_{\rho}+(u^{\rho}\partial_{\rho}v^{\phi}-u^3\partial_{\rho}(\rho v^{\phi})+\frac{u^{\phi}v^{\rho}}{\rho})e_{\phi}\\
    & +(u^{\rho}\partial_{\rho}v^3-u^3\partial_{\rho}(\rho v^3))e_3,
   \end{split}
\end{equation}
and
\begin{equation}\label{eqE_7}
\dive u=\frac{1}{\rho}\partial_{\rho}(\rho u^{\rho})-\partial_{\rho}(\rho u^3).
\end{equation}
Note $p$ is $(-2)$-homogeneous and axisymmetric, using (\ref{eqE_0_2}) with $\alpha=2$, we have
\[
  \nabla p=\partial_{\rho}p e_{\rho}-\frac{1}{\rho}\partial_{\rho}(\rho^2 p) e_3 \quad \textrm{ on }\{x_3=1\}.
\]
On $\{x_3=1\}$, the system  (\ref{NS}) for $(-1)$-homogeneous axisymmetric solutions   can be written in $(\rho, \phi, x_3)$ as
\begin{equation}\label{eqNC_C}
   \left\{
   \begin{split}
     & -(\Delta u^{\rho}-\frac{u^{\rho}}{\rho^2})+u^{\rho}\partial_{\rho}u^{\rho}-u^3\partial_{\rho}(\rho u^{\rho})-\frac{(u^{\phi})^2}{\rho}+\partial_{\rho}p=0,\\
      & -(\Delta u^{\phi}-\frac{u^{\phi}}{\rho^2})+u^{\rho}\partial_{\rho}u^{\phi}-u^3\partial_{\rho}(\rho u^{\phi})+\frac{u^{\phi}u^{\rho}}{\rho}=0,\\
      & -\Delta u^3+u^{\rho}\partial_{\rho}u^3-u^3\partial_{\rho}(\rho u^3)-\frac{1}{\rho}\partial_{\rho}(\rho^2p)=0,\\
      &
      \frac{1}{\rho}\partial_{\rho}(\rho u^{\rho})-\partial_{\rho}(\rho u^3)=0,
   \end{split}
   \right.
\end{equation}
where $\Delta u^j$ is given by (\ref{eqE_1}), $j=\rho, \phi, 3$.
 The relations between $(u^1, u^2)$ and $(u^{\rho}, u^{\phi})$ satisfy
 \begin{equation*}
 u^{\rho}=u_1\cos\phi+u_2\sin\phi, \quad u^{\phi}=-u_1\sin\phi+u_2\cos\phi.
	\end{equation*}
	So the estimates for the $\rho$-derivatives of $u^{\rho}, u^{\phi}$ imply the corresponding Cartesian $\nabla_{x'}$-estimates for $u'$, up to constants depending only on the order of differentiation.
	For simplicity, we  sometimes combine the first two equations in (\ref{eqNC_C}) as
\[
   -\Delta u'+u\cdot \nabla u'+\nabla_{x'}p=0, 
\]
where $\Delta u'$ denotes the planar vector part of the $3$-dimensional vector Laplacian. We will use the notation $\Delta'u'$ to denote $2$-dimensional vector Laplacian of $u'$.

\bigskip

 Next, we collect several auxiliary estimates used in the proof of Theorem \ref{thm_main}.
 The first is a regularity estimate for solutions of (\ref{NS}) with at most logarithmic growth near the $x_3$-axis. The proof is the same as that of Lemma 2.1 in \cite{LLY5}, with $o(\cdot)$ there replaced by $O(\cdot)$.
\begin{lem}\label{lemL_1}
	Let $R>0$ and $a<b$ satisfying $ab>0$,  and set $\Omega_{a, b, R}:=\{x\in\R^3\mid |x'|<R, a<x_3<b\}$.
	 Let $(u,p)\in C^{2}(\Omega_{a, b, R}\setminus\{x'=0\})$ be a solution of (\ref{NS}) satisfying
	\[
	  |u|\le C_0(|\ln |x'||+1), \quad x\in\Omega_{a, b, R}\setminus\{x'=0\},
	\]
	for some constant $C_0>0$.
	Then $(u, p)$ is smooth in $\Omega_{a, b, R}\setminus\{x'=0\}$. Moreover,  for any $a<a'<b'<b$ and $0<R'<R$, there exists $c_p\in \R$, which may depend on $p$, such that for every integer $k\ge 0$,
		\begin{equation*}
	|\nabla^ku|\le C\frac{|\ln |x'||+1}{|x'|^k},  \quad|\nabla^{k} (p-c_p)|\le C\frac{|\ln |x'||+1}{|x'|^{k+1}},
	\end{equation*}
	for all $x\in \Omega_{a', b', R'}\setminus\{x'=0\}$, where $C>0$ depends only on $R, R', a, b, a', b', k$ and $C_0$.
\end{lem}

For any $R>0$, let $D_R=\{x'\in \R^2\mid |x'|<R\}$ be the disk in $\R^2$ with radius $R$. Set $\rho=|x'|$  and $\Delta'=\partial_1^2+\partial_2^2$. We first consider the equation $\Delta' u=f$ for scalar functions $u$ and $f$ that are radially symmetric in the $x'$-variable. In this case,
 \[
   \Delta' u=\frac{1}{\rho}\partial_{\rho}(\rho\partial_{\rho}u).
 \]
 We nevertheless keep the notation $\Delta'$ for consistency with the recursive equations used later in Section \ref{sec_3}.

\begin{lem}\label{lem2_1}
Let $R>0$ and $f\in C(D_R\setminus\{0\})$ be radially symmetric, i.e. $f=f(\rho)$. Assume there exist constants $C_0>0$ and $\alpha\in \R\setminus\{-2\}$ such that  $|f(\rho)|\le C_0\rho^{\alpha}$ in $D_R\setminus\{0\}$. Then there exists a radially symmetric  solution $\bar{u}\in C^{2}(D_R\setminus\{0\})$ of
\[
\Delta' \bar{u}=f,\quad \textrm{ in }D_R\setminus\{0\},
\]
 such that, for $0\le k\le 2$,
\be\label{eq2_1_1}
|\partial_{\rho}^k\bar{u}|\le C\rho^{\alpha+2-k}, \quad
 \textrm{ in }D_R\setminus\{0\},
\ee
for a constant $C>0$ depending only on $\alpha$, $C_0$ and $R$.
Moreover, let $L\in\mathbb{N}$ and assume that $f\in C^L(D_R\setminus\{0\})$ satisfies
  $|\partial_{\rho}^kf|\le C_0\rho^{\alpha-k}$ for $0\le k\le L$, then $\bar{u}\in C^{L+2}(D_R\setminus\{0\})$ and (\ref{eq2_1_1}) holds for $0\le k\le L+2$, where $C$ depends on $\alpha, C_0, R$ and $L$.

    In particular, suppose $\alpha>\alpha_0$ for some fixed $\alpha_0>-2$. Let  $u$ be a radial solution of $\Delta' u=f$  in $D_R\setminus\{0\}$ satisfying $\lim_{\rho\to 0^+}u(\rho)=0$. If $|\partial_{\rho}^kf|\le C_0(|\alpha|+1)^k\rho^{\alpha-k}$ for $0\le k\le L$, then
 \be\label{eq2_1_2}
   |\partial_{\rho}^ku|\le DC_0(|\alpha|+1)^{k-2}\rho^{\alpha+2-k},  \textrm{ in }D_R\setminus\{0\},  \quad\forall 0\le k\le L+2,
 \ee
 for a constant $D>0$ depending only on $\alpha_0$ and $L$.
 	\end{lem}
	\begin{proof}
	For radially symmetric functions $u, f$, the equation
	$
	\Delta'u=f
	$
	  reduces to
	\be\label{eq2_1_3}
	  \frac{1}{\rho}\partial_{\rho}(\rho\partial_{\rho}u(\rho))=f(\rho).
	\ee
	Define
	\[
	  \bar{u}=\int_{\rho_0}^{\rho}\frac{1}{s}\int_{\rho_0}^{s}yf(y)dyds, \quad \textrm{ where }  \rho_0:=\left\{
	      \begin{array}{ll}
	         0 & \textrm{ if }\alpha>-2,\\
	         R & \textrm{ if }\alpha< -2.
	      \end{array}
	    \right.
	\]
	By (\ref{eq2_1_3}), direct computation  gives $\Delta'\bar{u}=f$, and (\ref{eq2_1_1}) for $0\le k\le 2$ follows from the above  integral representation.
	For higher order derivatives, differentiating (\ref{eq2_1_3}) repeatedly yields
	  \be\label{eq2_1_4}
	    \partial_{\rho}^{k+2}u=\sum_{i=0}^{k}c_{ki}\rho^{-1-k+i}\partial_{\rho}^{i+1}u+\partial_{\rho}^{k}f, \quad \forall k\ge 0,
	  \ee
	  for some constants $c_{ki}$, $0\le i\le k$. Then   (\ref{eq2_1_1}) for $3\le k\le L+2$ follows by induction.

	When $\alpha>-2$, the general solution of (\ref{eq2_1_3}) is
	\[
	  u=\int_{0}^{\rho}\frac{1}{s}\int_{0}^{s}yf(y)dyds+C_1\ln\rho+C_2.
	\]
	If $\lim_{\rho\to 0^+}u(\rho)=0$, then $C_1=C_2=0$.
	Note for $\alpha>\alpha_0>-2$, we have $C^{-1}(|\alpha|+1)\le \alpha+2\le C(|\alpha|+1)$ for some $C>0$ depending only on $\alpha_0$. So  (\ref{eq2_1_2}) holds for $0\le k\le 2$.
	The estimates for higher-order derivatives follow from (\ref{eq2_1_4}) by induction.
	\end{proof}

	 We next extend Lemma \ref{lem2_1} to axisymmetric vector fields.
	 A planar vector field $f(x')=f^{\rho}e_{\rho}+f^{\phi}e_{\phi}$, $x'\in\R^2\setminus\{0\}$, is said to be axisymmetric  if $f^{\rho}$ and $f^{\phi}$ are independent of $\phi$.  For an axisymmetric vector field $u'=u^{\rho}e_{\rho}+u^{\phi}e_{\phi}$, we have
	 \be\label{eqA_E_2}
	   \Delta' u'=(\frac{1}{\rho}\partial_{\rho}(\rho\partial_{\rho}u^{\rho})-\frac{u^{\rho}}{\rho^2})e_{\rho}+(\frac{1}{\rho}\partial_{\rho}(\rho\partial_{\rho}u^{\phi})-\frac{u^{\phi}}{\rho^2})e_{\phi}.
	 \ee
	\begin{lem}\label{lem2_2}
	Let $R>0$ and $f=f^{\rho}e_{\rho}+f^{\phi}e_{\phi}\in C(D_R\setminus\{0\})$ be an axisymmetric vector field. Assume there exist constants $C_0>0$ and $\alpha\in \R\setminus\{-1, -3\}$, such that  $|f|\le C_0\rho^{\alpha}$ in $D_R\setminus\{0\}$. Then there exists an axisymmetric solution
 $\bar{u}=\bar{u}^{\rho}e_{\rho}+\bar{u}^{\phi}e_{\phi}\in C^{2}(D_R\setminus\{0\})$ of
 \[
 \Delta' \bar{u}=f,\quad \textrm{ in }D_R\setminus\{0\},
 \]
 satisfying, for $0\le k\le 2$, 
 \be\label{eq2_2_1}
 |\partial_{\rho}^k\bar{u}|\le C\rho^{\alpha+2-k}, \quad
  \textrm{ in }D_R\setminus\{0\},
 \ee
 for a constant $C>0$ depending only on $\alpha$, $C_0$ and $R$.
 Moreover, let $L\in\mathbb{N}$ and assume that $f\in C^L(D_R\setminus\{0\})$ satisfies   $|\partial_{\rho}^kf|\le C_0\rho^{\alpha-k}$ for $0\le k\le L$.  Then $\bar{u}\in C^{L+2}(D_R\setminus\{0\})$ and  (\ref{eq2_2_1}) holds for $0\le k\le L+2$, where $C$ depends on $\alpha, C_0, R, L$.

    In particular, when $\alpha>\alpha_0$ for some fixed $\alpha_0>-1$ and $u$ is an axisymmetric solution of $\Delta' u=f$ satisfying $\lim_{\rho\to 0^+}u^j(\rho)=\lim_{\rho\to 0^+}\partial_{\rho}u^j(\rho)=0$,
    $j=\rho, \phi$, if $|\partial_{\rho}^kf^j|\le C_0(|\alpha|+1)^k\rho^{\alpha-k}$ for $0\le k\le L$, then
 \be\label{eq2_2_4}
   |\partial_{\rho}^ku|\le DC_0(|\alpha|+1)^{k-2}\rho^{\alpha+2-k}, \textrm{ in }D_R\setminus\{0\}, \quad \forall 0\le k\le L+2,
 \ee
 for a constant $D>0$ depending only on $\alpha_0$ and $L$.
	\end{lem}
	\begin{proof}
	By (\ref{eqA_E_2}), the equation $\Delta'u=f$ for axisymmetric vector field $u$ is equivalent to
		\be\label{eq2_2_2}
	  \frac{1}{\rho}\partial_{\rho}(\rho\partial_{\rho}u^{j})-\frac{u^{j}}{\rho^2}=f^{j}, \quad j=\rho, \phi.
	\ee
Define
	\[
	  \bar{u}^{j}=\rho\int_{\rho_0}^{\rho}r^{-3}\int_{r_0}^ry^2f^{j}(y)dydr, \quad j=\rho, \phi,
	  	\]
	where
	\[
	   r_0=\left\{
	       \begin{split}
	       & 0, \quad\textrm{ if }\alpha>-3,\\
	       &  R, \quad \textrm{ if }\alpha<-3,
	       \end{split}
	    \right., \quad  \rho_0=\left\{
	       \begin{split}
	       & 0, \quad\textrm{ if }\alpha>-1,\\
	       & R, \quad \textrm{ if }\alpha<-1.
	       \end{split}
	    \right.
	\]
	By (\ref{eq2_2_2}),
	direct computation  shows  that $\bar{u}$ satisfies $\Delta' \bar{u}=f$. 	The estimate (\ref{eq2_2_1}) for $0\le k\le 2$ follows from the above integral representation. 	 For higher order derivatives, differentiating (\ref{eq2_2_2}) repeatedly yields
	  \be\label{eq2_3_5}
	    \partial_{\rho}^{k+2}u^j=\sum_{i=0}^{k+1}c_{ki}\rho^{-2-k+i}\partial_{\rho}^{i}u^j+\partial_{\rho}^{k}f^j, \quad \forall k\ge 0, j=\rho, \phi,
	  \ee
	  with some constants $c_{ki}$, $0\le i\le k+1$. Then   (\ref{eq2_2_1}) for $3\le k\le L+2$ follows by induction.

	When $\alpha>-1$,
	the general solution of (\ref{eq2_2_2}) is
	\[
	   u^j=\rho\int_{0}^{\rho}r^{-3}\int_{0}^ry^2f^{j}(y)dydr+C_1\rho+C_2\rho^{-1}.
	\]
	If $\lim_{\rho\to 0^+}u^j(\rho)=\lim_{\rho\to 0^+}\partial_{\rho}u^j(\rho)=0$,
	$j=\rho, \phi$, then $C_1=C_2=0$. Note for $\alpha>\alpha_0>-1$, we have $C^{-1}(|\alpha|+1)\le \alpha+1\le \alpha+3\le C(|\alpha|+1)$ for some $C>0$ depending only on $\alpha_0$.
	 So (\ref{eq2_2_4}) holds for $0\le k\le 2$. The estimates for higher-order derivatives follow from (\ref{eq2_3_5}) by induction.
\end{proof}

\section{Proof of Theorem \ref{thm_main} and Proposition \ref{prop1_1}}
\label{sec_3}

 Let $(u, p)$ be a $(-1)$-homogeneous axisymmetric solution of (\ref{NS}) in $\Omega_{\delta}\setminus\{x'=0\}$ for some $\delta>0$, satisfying $|u|=o(|x'|^{-1})$ as $x'\to 0$ for each fixed $x_3>0$. By homogeneity, this is equivalent to $|u|=o(\rho^{-1})$ on $\{x_3=1\}$.   By Theorem 1.3 in \cite{LLY1},  $|u|=O(|\ln |x'||+1)$ as $x'\to 0$  for each fixed $x_3>0$.
  We first determine the leading-order terms of such a solution, and  then construct and classify such solutions as stated in Theorem \ref{thm_main}.

As in Section \ref{sec_2}, homogeneous functions in $\R^3$ will be identified with their traces on $\{x_3=1\}$.
  Derivatives involving $\partial_3$ are understood through this homogeneous extension and are computed through (\ref{eqE_0_1})-(\ref{eqE_0_3}).
  For convenience, set
  \[
    \Gamma_{\delta}:=\{x\in\R^3\mid 0<|x'|<\delta, x_3=1\}.
  \]
Recall  we denote  $\rho=|x'|$ on $\{x_3=1\}$ and $t=\ln \rho$. For any radially symmetric scalar function or axisymmetric vector field $f(x')$ on $D_{\delta}\setminus\{0\}\subset\R^2$,  the relation between Cartesian derivatives and radial derivatives gives
	\be\label{eq3_E_1}
	|\nabla_{x'}^kf|\le C(k)\sum_{i=0}^{k}\rho^{-k+i}|\partial_{\rho}^if|, \quad \forall k\ge 0, \quad 0<\rho<\delta,
	\ee
for some constant $C(k)>0$ depending only on $k$.

To establish the expansion of the solutions we study, we first extract their leading order terms.
\begin{lem}\label{lem3_1}
   Let $\delta>0$ and $(u, p)\in C^2(\Omega_{\delta}\setminus\{x'=0\})$ be a $(-1)$-homogeneous axisymmetric solution of (\ref{NS}) in $\Omega_{\delta}\setminus\{x'=0\}$    satisfying $|u|=o(|x'|^{-1})$ as $x'\to 0$ for each fixed $x_3>0$.
Then there exist unique  constants $c_0, d_0, b_0^{\rho}, b_0^{\phi}$,
and axisymmetric functions $h_0, r_0\in C^{\infty}(\Omega_{\delta}\setminus\{x'=0\})$, where $h_0$ is $(-1)$-homogeneous and $r_0$ is $(-2)$-homogeneous, such that on $\Gamma_{\delta}$,
\begin{equation}\label{eq3_1_0_1}
\begin{split}
&  u^{\rho}=\rho(\frac{1}{2}c_0t+b_0^{\rho})+h_0^{\rho}, \quad  u^{\phi}=b_0^{\phi}\rho+h_0^{\phi}, \quad
 u^3=c_0t+b_0^3+h_0^3, \quad p=c_0t+d_0+r_0,
 \end{split}
\end{equation}
with $b_0^3=2b_0^{\rho}-\frac{1}{2}c_0$. Moreover, for any $0<\delta'<\delta$, $k, l\in\mathbb{N}$ and any $\epsilon>0$, there exists a constant $C>0$, depending on $(u, p), \delta,\delta',  k, l, \epsilon$, such that
\be\label{eq3_1_0_2}
 |\nabla^k_{x'}\partial_{3}^lh_0'|\le C\rho^{3-k-\epsilon}, \quad  |\nabla^k_{x'}\partial_{3}^l h_0^3|\le C\rho^{2-k-\epsilon}, \quad |\nabla^k_{x'}\partial_{3}^l r_0|\le C\rho^{2-k-\epsilon},\quad \textrm{ on }\Gamma_{\delta'}.
 \ee
\end{lem}
\begin{proof}
Let $0<\delta'<\delta$ and let $(u, p)$ be a fixed $(-1)$-homogeneous axisymmetric solution of (\ref{NS}) in $\Omega_{\delta}\setminus\{x'=0\}$  satisfying $|u|=o(|x'|^{-1})$ as $x'\to 0$ for each $x_3>0$.
It suffices to prove the lemma for $0<\epsilon<1/2$. 
Throughout the proof, let $k, l\in\mathbb{N}$ and $0<\epsilon<1/2$ be  arbitrary in each estimate, and $C>0$ denote a constant depending only on $(u, p), \delta, \delta', k, l, \epsilon$, which may vary from line to line.
%Let $k, l\in\mathbb{N}$ and $0<\epsilon<1/2$ be  arbitrary.
%Unless otherwise stated, all estimates below hold for every $k,l\in\mathbb N$ and $0<\epsilon<1/2$.
%Throughout the proof, $C>0$ denotes a constant depending only on $(u, p), \delta, \delta', k, l, \epsilon$, which may vary from line to line.

\medskip

\noindent 1. By Theorem 1.3 in \cite{LLY1}, we have $|u|\le C(|\ln |x'||+1)$ in any sub-cylinder of $\Omega_{\delta}\setminus\{x'=0\}$.
	 Apply Lemma \ref{lemL_1} on compact sub-cylinders of $\Omega_{\delta}$ for the estimates near $\{x'=0\}$,  using the smoothness of solutions away from $\{x'=0\}$, and then using (\ref{eqE_0_4}),
	 we obtain, for all $k, l\in\mathbb{N},  0<\epsilon<1/2$, that 
	\begin{equation}\label{eq3_1_1}
	\begin{split}
	 &   |\nabla_{x'}^k\partial_3^lu|\le C\rho^{-k}(|\ln \rho|+1)\le C\rho^{-k-\epsilon}, \quad \textrm{ in }\Gamma_{\delta'}. %\forall k, l\in\mathbb{N},\  0<\epsilon<1/2,\x\in\Gamma_{\delta'}.
	   \end{split}
	\end{equation}
	We first show that this implies $|p|\le C(|\ln\rho|^3+1)$ on $\Gamma_{\delta'}$.
To see this,  write (\ref{NS}) in spherical coordinates $(R, \theta, \phi)$, where $R=|x|$, $\theta$ is the angle between $x$ and the positive $x_3$-axis, and $\phi$ is the azimuthal angle.  The spherical form of (\ref{NS}) gives (see e.g. \cite{LLY1} for the full system)
\[
    \frac{d}{d\theta} (\frac{1}{2}(u^{\theta})^2 - u^R + p) =  \cot\theta (u^{\phi})^2 \quad \textrm{ on }\bS^2.
\]
Let $\bar{\delta}=(\delta'+\delta)/2$ and set $\theta_0=\tan^{-1}\bar{\delta}$. We have
	\begin{equation}\label{eq_p}
	   p=C_0+u^R-\frac{1}{2}(u^{\theta})^2+\int_{\theta_0}^{\theta}(\cot \alpha) (u^{\phi})^2d\alpha, \quad \textrm{ on }\bS^2,
	\end{equation}
	where
	$C_0$ is a constant.
	Since $u$ is smooth away from $\{x'=0\}$ and $|u|\le  C(|\ln\sin\theta|+1)$  on $\bS^2\cap\{0<\theta<\theta_0\}$,
	we have
	$|p|\le C(|\ln \sin \theta|^3+1)$ on $\bS^2\cap\{0<\theta<\theta_0\}$. Thus for any $0<\epsilon<1/2$, 
	\[
	  |p|\le C(|\ln\rho|^3+1)\le C\rho^{-\epsilon}, \quad \textrm{ in }\Gamma_{\delta'}. %\forall 0<\epsilon<1/2,\   x\in\Gamma_{\delta'}.
	\]
	Differentiating (\ref{eq_p}) repeatedly in $\theta$,  using the estimates for $u$ and its derivatives in  (\ref{eq3_1_1}), and then converting $\theta$-derivatives to $x'$-derivatives and $x_3$-derivatives through homogeneity, we have, for all $k, l\in\mathbb{N}, 0<\epsilon<1/2$, that 
	\begin{equation}\label{eq3_1_2}
	  |\nabla^{k}_{x'} \partial_3^l p|\le C\rho^{-k}(|\ln \rho|^3+1)\le C\rho^{-k-\epsilon}, \quad \textrm{ in }\Gamma_{\delta'}.
	  %\forall k, l\in\mathbb{N},\  0<\epsilon<1/2,\x\in \Gamma_{\delta'}.
	\end{equation}

	\noindent 2. We first show that $u'$ has improved estimate near $\{x'=0\}$,
	and use it to extract the leading-order terms of $u^3$.
	Write the equation of $u'$ in (\ref{NS}) as
 \be\label{eq3_1_6}
   \Delta' u'=-\partial^2_3 u'+u\cdot \nabla u'+\nabla_{x'}p=:g_{-1}' \quad \textrm{ in }\Gamma_{\delta}.
 \ee
	By (\ref{eq3_1_1}) and (\ref{eq3_1_2}),  using  (\ref{eq3_1_1})  with $\epsilon$ replaced by $\epsilon/2$ in the nonlinear term, we have, for all $k\in\mathbb{N}$ and $0<\epsilon<1/2$, that 
	\[
	|\nabla^k_{x'}g_{-1}'|\le C\rho^{-k-1-\epsilon}, \quad \textrm{ in }\Gamma_{\delta'}. %\forall k\in\mathbb N,\ 0<\epsilon<1/2,\ x\in\Gamma_{\delta'}.
	\]
	Since $g_{-1}'$ is axisymmetric, by Lemma \ref{lem2_2} and (\ref{eq3_E_1}), there exists an axisymmetric vector field $\tilde{h}_{0}'(x')\in C^{\infty}(D_{\delta'}\setminus\{0\})$ such that
	\[
	  \Delta'\tilde{h}_0'(x')=g_{-1}'(x', 1), \quad \textrm{ in } D_{\delta'}\setminus\{0\},
	\]
	and for any $k\in\mathbb{N}$ and $0<\epsilon<1/2$, 
	\[
	%\Delta'\tilde{h}_0'(x')=g_{-1}'(x', 1), \textrm{ and }
	|\nabla^k_{x'}\tilde{h}_0'|\le C\rho^{1-k-\epsilon},
	 \quad  %\forall k\in\mathbb N,\ 0<\epsilon<1/2, \ 
	 \textrm{ in } D_{\delta'}\setminus\{0\}.
	\]
	 Let $\tilde{u}_0'(x')=u'(x', 1)-\tilde{h}_0'(x')$. Then
	 \[
	 \Delta' \tilde{u}_0'=0\quad \textrm{ in }D_{\delta'}\setminus\{0\}.
	 \]
	 Moreover,  $\tilde{u}_0'$ is axisymmetric and satisfies $|\tilde{u}_0'|\le C(|\ln\rho|+1)$. Note the system $\Delta' \tilde{u}_0'=0$ for an axisymmetric vector field reduces to
	 \be\label{eq3_1_8}
	    \frac{1}{\rho}\partial_{\rho}(\rho\partial_{\rho}\tilde{u}^{j}_0)-\frac{\tilde{u}_0^{j}}{\rho^2}=0, \quad j=\rho, \phi, \quad \textrm{ in }D_{\delta'}\setminus\{0\}.
	 \ee
	 Hence
	 \be\label{eq3_1_9}
	   \tilde{u}_0^j=C_1^j\rho+C_2^j\rho^{-1},\quad j=\rho, \phi,
	 \ee
	 for some constants $C_1^j, C_2^j$.
	 Since $|\tilde{u}_0'|\le C(|\ln\rho|+1)$, we have $C_2^j=0$.
	 Then $u'=\tilde{h}_0'+\tilde{u}_0'$ satisfies $|\nabla^k_{x'}u'|\le C\rho^{1-k-\epsilon}$ on $\Gamma_{\delta'}$ for any $k\in\mathbb{N}$ and $0<\epsilon<1/2$.
	 By (\ref{eqE_0_4}) and homogeneity of $u'$, we have, for all $k, l\in\mathbb{N}, 0<\epsilon<1/2$, that 
	 \be\label{eq3_1_3}
	    |\nabla^k_{x'}\partial_3^lu'|\le C\rho^{1-k-\epsilon},
	     \quad \textrm{ in }\Gamma_{\delta'}.%\forall k, l\in\mathbb{N},\ 0<\epsilon<1/2, \ x\in\Gamma_{\delta'}.
	 \ee
	Next, write the equation for $u^3$ in (\ref{NS})  as
	\[
	  \Delta' u^3=-\partial_3^2u^3+u\cdot \nabla u^3+\partial_3 p=:g_{-1}^3 \quad \textrm{ in }\Gamma_{\delta}.
	\]
	By (\ref{eq3_1_1}),  (\ref{eq3_1_2}) and (\ref{eq3_1_3}),  using the estimates with $\epsilon$ replaced by $\epsilon/2$ in the nonlinear term, we have, for any $k\in\mathbb{N}$ and $0<\epsilon<1/2$, that 
\[
  |\nabla^k_{x'} g_{-1}^3|\le C\rho^{-k-\epsilon}, \quad \textrm{ in }\Gamma_{\delta'}.%\forall k\in\mathbb N,\ 0<\epsilon<1/2,\ x\in\Gamma_{\delta'}.
\]
	Since $g_{-1}^3(x', 1)$ is radially symmetric, by Lemma \ref{lem2_1} and (\ref{eq3_E_1}), there exists a radially symmetric function $h_{0}^3(x')\in C^{\infty}(D_{\delta'}\setminus\{0\})$, such that for any $k\in\mathbb{N}$ and $0<\epsilon<1/2$, 
	\[
	\Delta'h_0^3(x')=g_{-1}^3(x', 1),  \textrm{ and }|\nabla^k_{x'}h_0^3|\le C\rho^{2-k-\epsilon}, \quad\textrm{ in }D_{\delta'}\setminus\{0\}.
	%\forall k\in\mathbb{N},\ 0<\epsilon<1/2,
	%\ x'\in D_{\delta'}\setminus\{0\}.
	\]
	 Let $u_0^3(x')=u^3(x', 1)-h_0^3(x')$. Then
	 \[
	 \Delta' u_0^3=0\quad \textrm{ in }D_{\delta'}\setminus\{0\}.
	 \]
	  Moreover, $u_0^3$ is radially symmetric and satisfies $|u_0^3|\le C(|\ln\rho|+1)$. Hence
	\[
	  u_0^3=a_0^3 t+b_0^3
	\]
	for some constants $a_0^3, b_0^3$. Define $u_0^3=a_0^3 t+b_0^3$ on $\Gamma_{\delta}$ and set $h_0^3=u^3-u_0^3$ there. Extend $h_0^3$ and $u_0^3$ to be $(-1)$-homogeneous functions in $\Omega_{\delta}\setminus\{x'=0\}$. By (\ref{eqE_0_4}) and homogeneity of $h_0^3$, we have,  for all $k, l\in\mathbb{N}, 0<\epsilon<1/2$, that 
	\[
	|\partial_3^l\nabla^k_{x'}h_0^3|\le C\rho^{2-k-\epsilon}, \quad \textrm{ in }\Gamma_{\delta'}. %\forall k, l\in\mathbb{N},\ 0<\epsilon<1/2,\ x\in \Gamma_{\delta'}.
	\]
	Thus $u^3=a_0^3 t+b_0^3+h_0^3$ on $\Gamma_{\delta}$, and the corresponding estimate in (\ref{eq3_1_0_2}) holds for $h_0^3$. The constants $a_0^3$ and $b_0^3$ will be determined later.

	\medskip

	\noindent 3. We now determine the leading-order terms of $p$. Taking divergence of the momentum equation in (\ref{NS}), we have
	\[
	   -\Delta' p=\partial_3^2p+\dive (u\cdot \nabla u)=:g_0^p \quad \textrm{ in }\Gamma_{\delta}.
	\]
	Denote $W=u\cdot \nabla u=W'+W_3e_3$, where $W'=u\cdot \nabla u'$ and $W_3=u\cdot \nabla u^3$. By (\ref{eq3_1_1}) and  (\ref{eq3_1_3}) with $\epsilon$ replaced by $\epsilon/2$, we have, for any $k\in\mathbb{N}$ and $0<\epsilon<1/2$, that
	\begin{equation}\label{eq3_1_4}
  |\nabla^k_{x'}W'| \le C\rho^{1-k-\epsilon}, \quad |\nabla^k_{x'}\partial_3W_3|\le C\rho^{-k-\epsilon},  \quad x\in\Gamma_{\delta'}.
\end{equation}
Then
\[
  |\nabla_{x'}^k\dive W|\le |\nabla_{x'}^{k+1}W'|+|\nabla_{x'}^{k}\partial_3W^3|\le C\rho^{-k-\epsilon}, \quad x\in \Gamma_{\delta'}.
\]
By (\ref{eq3_1_2}), we also have $|\nabla^k_{x'}\partial^2_3p|\le C\rho^{-k-\epsilon}$ on $\Gamma_{\delta'}$. So for any $k\in\mathbb{N}$ and $0<\epsilon<1/2$, 
\begin{equation*}
|\nabla^k_{x'}g_0^p|\le C(|\nabla_{x'}^k\partial^2_3p|+|\nabla_{x'}^k\dive W|)\le C\rho^{-k-\epsilon}, \quad \textrm{ in }\Gamma_{\delta'}. %\forall k\in\mathbb{N},\ 0<\epsilon<1/2,\ x\in \Gamma_{\delta'}.
\end{equation*}
 Since $g_0^p(x', 1)$ is radially symmetric, by Lemma \ref{lem2_1} and (\ref{eq3_E_1}), there exists a radially symmetric function $r_0(x')\in C^{\infty}(D_{\delta'}\setminus\{0\})$, such that for any $k\in\mathbb{N}$ and $0<\epsilon<1/2$, 
 \begin{equation}\label{eq3_1_5}
 -\Delta' r_0(x')=g_0^p(x', 1), \textrm{ and } |\nabla^k_{x'} r_0|\le C\rho^{2-k-\epsilon}, \quad \textrm{ in }D_{\delta'}\setminus\{0\}. % \forall k\in\mathbb{N},\ 0<\epsilon<1/2, \ x'\in D_{\delta'}\setminus\{0\}.
 \end{equation}
 Let $p_0(x')=p(x', 1)-r_0(x')$.  Then $\Delta' p_0=0$   in $D_{\delta'}\setminus\{0\}$. Moreover,  $p_0$ is radially symmetric, and satisfies $|p_0|\le C(|\ln \rho|^3+1)$. Hence
 \[
 p_0=c_0t+d_0
\]
  for some constants $c_0, d_0\in \R$. Define $p_0=c_0t+d_0$ on $\Gamma_{\delta}$, and set $r_0=p-p_0$ there. Extending $r_0$ and $p_0$ to be $(-2)$-homogeneous functions in $\Omega_{\delta}\setminus\{x'=0\}$, and using (\ref{eqE_0_4}), we have, for all $k, l\in\mathbb{N},  0<\epsilon<1/2$, that 
  \[
  |\partial_3^l\nabla^k_{x'} r_0|\le C\rho^{2-k-\epsilon}, \quad \textrm{ in }\Gamma_{\delta'}. %\forall k, l\in\mathbb{N},  \ 0<\epsilon<1/2,\  x\in \Gamma_{\delta'}.
  \]
 Thus (\ref{eq3_1_0_1}) and (\ref{eq3_1_0_2}) hold for $p$ and $r_0$.

 \medskip

 \noindent 4. Now we show $u'$ can be decomposed as in (\ref{eq3_1_0_1}) with $h_0'$ satisfying (\ref{eq3_1_0_2}).
 Let
 \[
 w_0'(x'):=w_0^{\rho}e_{\rho}+w_0^{\phi}e_{\phi}, \quad \textrm{where } w_0^{\rho}=\frac{1}{2}c_0\rho t \textrm{ and } w_0^{\phi}=0.
 \]
  By direct computation, we have
 \[
   \Delta' w'_0(x')=\frac{c_0}{\rho}e_{\rho}=\nabla_{x'}p_0(x', 1), \quad \textrm{ in }D_{\delta}\setminus\{0\}.
 \]
 Extend $w_0'$ to be a $(-1)$-homogeneous vector field in $\Omega_{\delta}\setminus\{x'=0\}$. Let $v_0'=u'-w_0'$. Subtracting the above from (\ref{eq3_1_6}), and using $p=p_0+r_0$, we have
 \[
    \Delta' v_0'=-\partial^2_3 u'+u\cdot \nabla u'+\nabla_{x'}r_0=:g_0' \quad \textrm{ in }\Gamma_{\delta}.
 \]
 By (\ref{eq3_1_3}) and (\ref{eq3_1_4}), using the estimates with $\epsilon$ replaced by $\epsilon/2$ in the nonlinear term, as well as (\ref{eq3_1_5}) with $k$ replaced by $k+1$,
 we have, for any $k\in\mathbb{N}$ and $0<\epsilon<1/2$, that 
 \[
 |\nabla_{x'}^kg_0'|\le C\rho^{1-k-\epsilon}, \quad \textrm{ in }\Gamma_{\delta'}.%\forall k\in\mathbb{N},\ 0<\epsilon<1/2,\ x\in\Gamma_{\delta'}.
 \]
   Since $g_0'$ is axisymmetric, by Lemma \ref{lem2_2} and (\ref{eq3_E_1}), there exists an axisymmetric vector-valued function $h_0'(x')\in C^{\infty}(D_{\delta'}\setminus\{0\})$, satisfying, for any $k\in \mathbb{N}$ and $0<\epsilon<1/2$, that
 \begin{equation}\label{eq3_1_7}
\Delta' h_0'(x')=g_0'(x', 1), \textrm{ and }  |\nabla^k_{x'} h_0'|\le C\rho^{3-k-\epsilon},\quad \textrm{ in }D_{\delta'}\setminus\{0\}.
%\ x'\in D_{\delta'}\setminus\{0\}.
 \end{equation}
 Let $\tilde{v}_0'(x')=v_0'(x', 1)-h_0'(x')=u'(x', 1)-w_0'(x')-h_0'(x')$. Then $\Delta' \tilde{v}_0'=0$ in $D_{\delta'}\setminus\{0\}$ and   $\tilde{v}_0'$ is axisymmetric.  By Step 2 and the above, we have
 \[
 |u'(x', 1)|\le C\rho^{1-\epsilon}, \quad |w_0'|\le C\rho^{1-\epsilon}, \quad |h_0'(x')|\le C\rho^{3-\epsilon},  \quad \textrm{ in }D_{\delta'}\setminus\{0\}.%x'\in D_{\delta'}\setminus\{0\}.
 \]
  So  $|\tilde{v}_0'|\le C\rho^{1-\epsilon}$ in $D_{\delta'}\setminus\{0\}$. By the computation in (\ref{eq3_1_8})-(\ref{eq3_1_9}),
 this implies
 \[
 \tilde{v}_0^{\rho}=b_0^{\rho}\rho, \quad \tilde{v}_0^{\phi}=b_0^{\phi}\rho
\]
  for some constants $b_0^{\rho}, b_0^{\phi}\in \R$. Define
  \[
  u_0'(x')=w_0'(x')+\tilde{v}_0'(x')=\rho(\frac{1}{2}c_0t+b_0^{\rho})e_{\rho}+b_0^{\phi}\rho e_{\phi},\quad h_0'=u'-u_0', \quad \textrm{ on }\Gamma_{\delta}.
  \]
   Extend $u_0'$ and $h_0'$ to be $(-1)$-homogeneous vector fields in $\Omega_{\delta}\setminus\{x'=0\}$. By (\ref{eqE_0_4}) and (\ref{eq3_1_7}), we have, for all $k, l\in\mathbb{N}, 0<\epsilon<1/2$, that 
  \[
  |\partial_3^l\nabla^k_{x'} h_0'|\le C\rho^{3-k-\epsilon}, \quad \textrm{ in }\Gamma_{\delta'}. %\forall k, l\in\mathbb{N},\ 0<\epsilon<1/2,\ x\in \Gamma_{\delta'}.
  \]
  Thus $u'=u_0'+h_0'$ on $\Gamma_{\delta}$, and (\ref{eq3_1_0_1}) and (\ref{eq3_1_0_2}) hold for $u'$ and $h_0'$.

  \medskip

  \noindent 5. Finally, we use the divergence-free condition to determine $a_0^3$ and $b_0^3$.
    By (\ref{eq3_1_0_2}) and the divergence formula (\ref{eqE_7}), we have $|\dive h_0|\le C\rho^{2-\epsilon}$      on $\Gamma_{\delta'}$.
  Insert $u=u_0+h_0$ into $\dive u=0$, where $u_0^3=a_0^3t+b_0^3$, $u_0^{\rho}=\rho(\frac{1}{2}c_0t+b_0^{\rho})$ and $u_0^{\phi}=b_0^{\phi}\rho$, we have
  \[
    0= \dive u=\dive u_0+\dive h_0=\frac{1}{2}c_0+2(\frac{1}{2}c_0t+b_{0}^{\rho})-a_0^3-(a_0^3t+b_0^3)+O(\rho^{2-\epsilon}), \quad  \textrm{in }\Gamma_{\delta'}.
  \]
  Hence  $a_0^3=c_0$ and $b_0^3=2b_0^{\rho}-\frac{1}{2}c_0$. This implies $\dive u_0=0$. 
  
  The uniqueness of $c_0, d_0, b_0^{\rho}, b_0^{\phi}$ follows from (\ref{eq3_1_0_1}) and (\ref{eq3_1_0_2}) by sending $\rho\to 0$, and the uniqueness of $h_0, r_0$ then follows from (\ref{eq3_1_0_1}). The proof is finished.
\end{proof}

We now construct the formal expansion of a $(-1)$-homogeneous axisymmetric solution of (\ref{NS})  on the cross section $\{x_3=1\}\setminus\{x'=0\}$ with at most  Type II  singular behavior, associated with arbitrary parameters $(c_0, d_0, b_0^{\rho}, b_0^{\phi})\in\R^4$. Let $b_0^3=2b_0^{\rho}-\frac{1}{2}c_0$ and define
\begin{equation}\label{eq3_u_0}
    u_0^{\rho}=\rho(\frac{1}{2}c_0t+b_0^{\rho}), \ u_0^{\phi}=b_0^{\phi}\rho, \ u_0^3=c_0t+b_0^3, \ p_0=c_0t+d_0, \ \textrm{ on }\{x_3=1\}\setminus\{x'=0\}.
\end{equation}
We construct a sequence $\{(u_n, p_n)\}_{n=0}^{\infty}$ satisfying the recursive system (\ref{eq1_3}), such that each $u_n$ is $(-1)$-homogeneous, each $p_n$ is $(-2)$-homogeneous, and
 \begin{equation}\label{eq3_u}
 u^{\rho}_n=\rho^{2n+1}a^{\rho}_{n}(t), \quad u^{\phi}_n=\rho^{2n+1}a^{\phi}_{n}(t), \quad u^{3}_n=\rho^{2n}a^3_{n}(t), \quad p_n=\rho^{2n}c_{n}(t), \quad n\ge 0,
   \end{equation}
 on $\{x_3=1\}\setminus\{x'=0\}$, where $c_{n}(t)$ and $a^{j}_{n}(t)$, $j=\rho, \phi, 3$, are polynomials of $t$ satisfying (\ref{eq1_A2}).  In the following, we denote
\[
  W_n:=\sum_{m=0}^nu_m\cdot \nabla u_{n-m}, \quad W_{nm}:=u_m\cdot \nabla u_{n-m}.
\]
 \begin{lem}\label{lem3_2}
    Let $c_0, d_0, b_0^{\rho}, b_0^{\phi}\in \R$, and let $u_0$ be the $(-1)$-homogeneous vector field and $p_0$ be the $(-2)$-homogeneous function       defined by (\ref{eq3_u_0}) on $\{x_3=1\}\setminus\{ x'= 0\}$ with $b_0^3=2b_0^{\rho}-\frac{1}{2}c_0$. Then there exists a unique sequence  of axisymmetric functions $\{(u_n, p_n)\}_{n=0}^{\infty}$, with prescribed initial term $(u_0, p_0)$,       such that each $u_n$ is $(-1)$-homogeneous, each $p_n$ is $(-2)$-homogeneous,
   and $(u_n, p_n)$ has the form (\ref{eq3_u}) for every $n\ge 0$, and the sequence satisfies (\ref{eq1_3}) on $\{x_3=1\}\setminus\{x'=0\}$. Moreover,     $\dive u_n=0$     and (\ref{eq1_A2}) holds
      for all $n\ge 0$.
\end{lem}
\begin{proof}
We construct the sequence recursively on $\{x_3=1\}\setminus\{x'=0\}$. By definition, $(u_0, p_0)$ is in the form (\ref{eq3_u}) and satisfies  (\ref{eq1_A2}) for $n=0$. A direct computation using $b_0^3=2b_0^{\rho}-\frac{1}{2}c_0$ gives $\dive u_0=0$.
Let $n\ge 0$, and assume that $(u_m, p_m)$ have been constructed for all $0\le m\le n$, where each $u_m$ is $(-1)$-homogeneous, each $p_m$ is $(-2)$-homogeneous, and $(u_m, p_m)$ is in the form (\ref{eq3_u}) on $\{x_3=1\}\setminus\{x'=0\}$, satisfying (\ref{eq1_A2}).  We also assume that the recursive system (\ref{eq1_3}) on $\{x_3=1\}\setminus\{x'=0\}$ has been satisfied for indices $0, ..., n-1$.
We now construct $(u_{n+1}, p_{n+1})$ in the form (\ref{eq3_u}) on $\{x_3=1\}\setminus\{x'=0\}$ so that  (\ref{eq1_3}) holds for index $n$.

\medskip

\noindent 1. Recall we denote $W_n=\sum_{m=0}^nu_m\cdot \nabla u_{n-m}$.
By the form (\ref{eq3_u}) for $u_m$, $0\le m\le n$,  and (\ref{eqE_6}),  direct computation gives
  \[
    W_n^{\rho}=\rho^{2n+1}B_{n}^{\rho}(t), \quad W_n^{\phi}=\rho^{2n+1}B_{n}^{\phi}(t),\quad W_n^3=\rho^{2n}B_{n}^3(t),\quad \textrm{ on }\{x_3=1\}\setminus\{x'=0\},
        \]
    where  each $B_n^j(t)$ is a polynomial, given by
  \begin{equation}\label{eq3_2_1}
    \begin{split}
        & B_{n}^{\rho}(t)=\sum_{m=0}^n\left((a_{m}^{\rho}-a_{m}^3)((a_{n-m}^{\rho})'+(2(n-m)+2)a_{n-m}^{\rho})
   -a_{m}^{\rho}a_{n-m}^{\rho}-a_{m}^{\phi}a_{n-m}^{\phi})\right), \\
     & B_{n}^{\phi}(t)=\sum_{m=0}^n\left((a_{m}^{\rho}-a_{m}^3)((a_{n-m}^{\phi})'+(2(n-m)+2)a_{n-m}^{\phi})\right),\\
   & B_{n}^3(t)=\sum_{m=0}^n\left((a_{m}^{\rho}-a_{m}^3)((a_{n-m}^{3})'+(2(n-m)+1)a_{n-m}^{3})-a_{m}^{\rho}a_{n-m}^3\right).
   \end{split}
    \end{equation}
    In the formula for $B_{n}^{\phi}$, the terms $-a_m^{\rho}a_{n-m}^{\phi}+a_m^{\phi}a_{n-m}^{\rho}$ cancel after summing over $m$. By induction hypothesis,  in (\ref{eq3_2_1}) we have
     $\deg a_m^j\le m+1$ and $\deg a_{n-m}^j\le n-m+1$, $j=\rho, \phi, 3$.
     Since each $B_n^j$ is a finite sum of products of such polynomials and their derivatives, we have
    \begin{equation}\label{eq4_I_7}
    \deg B_{n}^{j}\le n+2,     \quad j=\rho, \phi, 3.
    \end{equation}

   \noindent 2.  Now we solve the first equation in (\ref{eq1_3}) for  $p_{n+1}$ in the form of (\ref{eq3_u}). The equation is
         \be\label{eq3_2_p}
      -\Delta' p_{n+1}=f_{n}^p=\partial_3^2p_n+\dive W_n, \quad \textrm{ on }\{x_3=1\}\setminus\{x'=0\}.
    \ee
    Since $W_n$ is $(-3)$-homogeneous, it follows from (\ref{eqE_0_1}) with $\alpha=3$ that
    \[
      \dive W_n=\frac{1}{\rho}\partial_{\rho}(\rho W_n^{\rho})+\partial_3W_n^3=\frac{1}{\rho}\partial_{\rho}(\rho W_n^{\rho})-\frac{1}{\rho^2}\partial_{\rho}(\rho^3W_n^3)=\rho^{2n}D_{n}(t),
    \]
    where
    \[
      D_{n}=\partial_tB_{n}^{\rho}+(2n+2)B_{n}^{\rho}-\partial_tB^3_{n}-(2n+3)B^3_{n}.
    \]
    Since $p_n$ is $(-2)$-homogeneous, using (\ref{eqE_0_1}) with $\alpha=2$ and the above, we have
        \[
        f_n^p=\partial_3^2p_n+\dive W_n=\rho^{2n}C_{n}(t), \quad \textrm{ on }\{x_3=1\}\setminus\{x'=0\},
        \]
         where
     \begin{equation}\label{eq3_2_2}
      C_{n}(t)=(\partial_t^2+(4n+5)\partial_t+(2n+2)(2n+3))[c_n(t)]+D_n.
    \end{equation}
    Let $p_{n+1}=\rho^{2(n+1)}c_{n+1}(t)$ with a polynomial $c_{n+1}(t)$. Then (\ref{eq3_2_p})
     is reduced to
    \be\label{eq3_2_3}
           -(\partial_t^2+4(n+1)\partial_t+4(n+1)^2)[c_{n+1}(t)]=C_n(t).
               \ee
               The zero-order coefficient of the differential operator on the left-hand side is $4(n+1)^2\ne 0$. For every $d\ge 0$, the operator is invertible on the space of polynomials of degree at most $d$.
    Thus there exists a unique polynomial solution $c_{n+1}(t)$. %, and $\deg c_{n+1}\le \deg C_n$. 
    
    Since $D_n$ is obtained from $B_n^j$ and their derivatives, by (\ref{eq4_I_7}) we have $\deg D_{n}\le n+2$. By the induction assumption, $\deg c_n\le n+1$. By this and (\ref{eq3_2_2}) we have $\deg C_n\le n+2$.  Hence the polynomial solution $c_{n+1}$  of  (\ref{eq3_2_3}) satisfies  $\deg c_{n+1}\le \deg C_n\le n+2$.

\medskip

    3. Next, we solve the second equation in (\ref{eq1_3}) for $u'_{n+1}=u^{\rho}_{n+1}e_{\rho}+u^{\phi}_{n+1}e_{\phi}$ with $u_{n+1}^{\rho}, u_{n+1}^{\phi}$ of the form of (\ref{eq3_u}). The equation is
      \[
      \Delta' u'_{n+1}=f_n'=-\partial_3^2u_n'+W_n'+\nabla_{x'}p_{n+1}, \quad \textrm{ on }\{x_3=1\}\setminus\{x'=0\}.
    \]
    Write $f_n'=f_n^{\rho}e_{\rho}+f_n^{\phi}e_{\phi}$ and the above equations as
    \be\label{eq3_2_u'}
     \frac{1}{\rho}\partial_{\rho}(\rho\partial_{\rho}u^j_{n+1})-\frac{u_{n+1}^{j}}{\rho^2}=f^{j}_n, \quad j=\rho, \phi.
         \ee
     By the homogeneity of $u_n$ and $p_{n+1}$, together with (\ref{eqE_0_1}) and (\ref{eq3_2_1}), we have
 \[
   f_n^{\rho}=-\partial_3^2u_n^{\rho}+W_n^{\rho}+\partial_{\rho}p_{n+1}=\rho^{2n+1}A_n^{\rho}(t), \quad f_n^{\phi}=-\partial_3^2u_n^{\phi}+W_n^{\phi}=\rho^{2n+1}A_n^{\phi}(t),
 \]
 where
     \begin{equation}\label{eq3_2_6}
    \begin{split}
       & A_n^{\rho}(t)= -(\partial_t^2+(4n+5)\partial_t+(2n+2)(2n+3))[a_n^{\rho}(t)]+B_{n}^{\rho}+(\partial_t+2n+2)[c_{n+1}(t)], \\
       & A_n^{\phi}(t)=-(\partial_t^2+(4n+5)\partial_t+(2n+2)(2n+3))[a_n^{\phi}(t)]+B_{n}^{\phi}.
      \end{split}
    \end{equation}
    Let $u^{j}_{n+1}=\rho^{2(n+1)+1}a_{n+1}^{j}(t)$ with polynomials $a_{n+1}^j$, $j=\rho, \phi$. Then the equations in (\ref{eq3_2_u'})   are reduced to
    \be\label{eq3_2_7}
     (\partial_t^2+2(2n+3)\partial_t+(2n+3)^2-1)[a_{n+1}^j]=A_n^j(t), \quad j=\rho, \phi.
    \ee
    Since $(2n+3)^2-1\ne 0$, the same argument as for (\ref{eq3_2_3}) shows that, for each $j=\rho, \phi$, there exists a unique polynomial solution $a_{n+1}^j$ of (\ref{eq3_2_7}). %, and $\deg a^j_{n+1}\le \deg A^j_n$.

By Step 2, we have $\deg c_{n+1}\le n+2$. By this, the induction assumption that $\deg a_n^{j}\le n+1$, $j=\rho, \phi$, together with (\ref{eq3_2_6}) and (\ref{eq4_I_7}), we have $\deg A_n^{j}\le n+2$ for $j=\rho, \phi$. So %Then by (\ref{eq3_2_7}) we have
  $\deg a^{j}_{n+1}\le \deg A_n^j\le n+2$ for $j=\rho, \phi$.

\medskip

   4. Now we solve the third equation in (\ref{eq1_3}) for $u^3_{n+1}$ of the form of (\ref{eq3_u}). The equation is
    \be\label{eq3_2_u3}
      \Delta' u^3_{n+1}=f_{n}^3=-\partial_3^2u_n^3+W_n^3+\partial_3p_n, \quad \textrm{ on }\{x_3=1\}\setminus\{x'=0\}.
    \ee
    By the homogeneity of $u_n$ and $p_n$, together with (\ref{eqE_0_1}) and (\ref{eq3_2_1}), we have
    \[
    f_n^3
    =\rho^{2n}A_n^3(t), \quad \textrm{ on }\{x_3=1\}\setminus\{x'=0\},
    \]
     where
     \begin{equation}\label{eq3_2_4}
      A_n^3(t)=-(\partial_t^2+(4n+3)\partial_t+(2n+2)(2n+1))[a_n^3(t)]+B_{n}^3-(\partial_t+2n+2)[c_{n}(t)].
    \end{equation}
    Let $u^3_{n+1}=\rho^{2(n+1)}a^3_{n+1}(t)$ with a polynomial $a^3_{n+1}(t)$. Then (\ref{eq3_2_u3})  is reduced to
    \be\label{eq3_2_5}
           (\partial_t^2+4(n+1)\partial_t+4(n+1)^2)[a^3_{n+1}(t)]=A^3_n(t).
    \ee
    Since $4(n+1)^2\ne 0$, the same argument as for (\ref{eq3_2_3}) shows that (\ref{eq3_2_5}) has a unique polynomial solution $a^3_{n+1}(t)$. By the induction assumption and (\ref{eq4_I_7}), we have $\deg A^3_n\le n+2$,  and thus $\deg a^3_{n+1}\le \deg A^3_n\le n+2$.

    So we have constructed $(u_{n+1}, p_{n+1})$ of the form (\ref{eq3_u}) satisfying (\ref{eq1_A2}) and (\ref{eq1_3}) for index $n$ on $\{x_3=1\}\setminus\{x'=0\}$. Extend $u_{n+1}$ to be $(-1)$-homogeneous and $p_{n+1}$ to be $(-2)$-homogeneous in $\{x_3>0\}\setminus\{x'=0\}$.
 By induction, this constructs a unique sequence $\{(u_n, p_n)\}_{n=0}^{\infty}$ in the form (\ref{eq3_u}) satisfying (\ref{eq1_3}) on $\{x_3=1\}\setminus\{x'=0\}$.

      \medskip

  \noindent 5.   We show the sequence constructed above satisfies $\dive u_n=0$
    for all $n\ge 0$. By the definition of $u_0$ and the relation $b_0^3=2b_0^{\rho}-\frac{1}{2}c_0$, we have $\dive u_0=0$.
    Assume $\dive u_m=0$ for all $0\le m\le n$. Taking $\nabla_{x'}\cdot$ of the second equation and $\partial_3$ of the third equation of (\ref{eq1_3}) and adding them together, using the definitions of $f_n'$ and $f_n^3$, the first equation in (\ref{eq1_3}) and the induction hypothesis $\dive u_n=0$,  we have
\[
\begin{split}
  \Delta' \dive u_{n+1} & =\dive \Delta' u_{n+1}  =\nabla_{x'}\cdot f_n'+\partial_3f_n^3
      =-\partial_3^2\dive u_n+\dive W_n+\partial^2_3p_n+\Delta'p_{n+1}
            =0,
      \end{split}
\]
   on $\{x_3=1\}\setminus\{x'=0\}$. Hence $\dive u_{n+1}(x', 1)$ is a harmonic function in $\R^2\setminus\{0\}$.    On the other hand, $u_{n+1}(x', 1)$ is in the form (\ref{eq3_u}), where
   \[
   u_{n+1}^{j}(x', 1)=\rho^{2n+3}a_{n+1}^{j}(t), \quad j=\rho, \phi, \quad u_{n+1}^{3}(x', 1)=\rho^{2n+2}a_{n+1}^{3}(t).
   \]
   Since $u_{n+1}$ is $(-1)$-homogeneous, by (\ref{eqE_7}), we have
      \begin{equation*}
   \begin{split}
   \dive u_{n+1}(x', 1) & =\frac{1}{\rho}\partial_{\rho}(\rho u_{n+1}^{\rho})-\partial_{\rho}(\rho u_{n+1}^3)\\
    & =\rho^{2n+2}(\partial_ta_{n+1}^{\rho}+(2n+4)a_{n+1}^{\rho}-\partial_ta_{n+1}^{3}-(2n+3)a_{n+1}^{3}).
   \end{split}
   \end{equation*}
  Thus $\dive u_{n+1}(x', 1)=\rho^{2n+2}Q(t)$ for some polynomial $Q(t)$.
    Since $\dive u_{n+1}(x', 1)$  is radial and harmonic in $\R^2\setminus\{0\}$, we also have $\dive u_{n+1}(x', 1)=C_1\ln\rho+C_2$ for some constants $C_1, C_2$.  So $\rho^{2n+2}Q(t)=C_1t+C_2$ with $n\ge 0$. Sending $\rho\to 0$, equivalently $t\to -\infty$, we obtain $C_1=C_2=0$ and hence $\dive u_{n+1}(x', 1)=0$.
        By homogeneity, $\dive u_{n+1}=0$ in $\{x_3>0\}\setminus\{x'=0\}$. 
        The proof is finished.
\end{proof}

Next, we show that the infinite sum of $\{(u_n, p_n)\}_{n=0}^{\infty}$ constructed in Lemma \ref{lem3_2} converges near $x'=0$ on $\{x_3=1\}$.
After homogeneous extension, it defines a local $(-1)$-homogeneous axisymmetric solution of (\ref{NS}) near the positive $x_3$-axis.

\begin{lem}\label{lem4_M_1}
Let   $K\subset \R^4$ be compact. For each $\lambda=(c_0, d_0, b_{0}^{\rho}, b_{0}^{\phi})\in K$,
let  $\{(u_{n}, p_{n})\}_{n=0}^{\infty}$ be the associated sequence constructed by Lemma \ref{lem3_2}.

\noindent (i) There exists $\delta>0$, depending only on $K$,  such that for every compact set $V\subset\subset \Omega_{\delta}\setminus\{x'=0\}$,  the series
  \[
    u_{\lambda}=\sum_{n=0}^{\infty}u_{n}, \quad p_{\lambda}=\sum_{n=0}^{\infty}p_{n},
  \]
  converge in $C^2(V)$ uniformly for $\lambda\in K$,
    and $(u_{\lambda}, p_{\lambda})$ is a smooth $(-1)$-homogeneous axisymmetric solution of  (\ref{NS}) in $\Omega_{\delta}\setminus\{x'=0\}$, satisfying $|u_{\lambda}|=o(|x'|^{-1})$ as $x'\to 0$ for each fixed $x_3>0$, and (\ref{eq1_a_3}).

       \medskip

     \noindent (ii)
     For the $\delta$ given in (i) and every integer $L\ge 2$,  there exists a constant $C>0$, depending only on $L$ and $K$, such that for any multi-index $\beta\in \mathbb{N}^4$  %and $k\in \mathbb{N}$ satisfying $0\le |\beta|+k\le L$, %
     satisfying $|\beta|\le L$, any $0\le k\le L$, 
     and any $x\in \Gamma_{\delta}$,
        \be\label{eqprop4_3}
         |\partial_{\rho}^k\partial_{\lambda}^{\beta}u'_{\lambda}|
         \le \left\{
         \begin{array}{ll}
          C\rho^{-k+1}(|\ln\rho|+1), &  \textrm{if }k=0, 1, \\
          C\rho^{-k+1}, & \textrm{if }k\ge 2,
         \end{array}
         \right.
       \ee
       \be\label{eqprop4_2}
         |\partial_{\rho}^k\partial_{\lambda}^{\beta}u^3_{\lambda}|+|\partial_{\rho}^k\partial_{\lambda}^{\beta}p_{\lambda}|
                  \le \left\{
         \begin{array}{ll}
          C(|\ln\rho|+1), &  \textrm{if }k=0, \\
          C\rho^{-k}, & \textrm{if }k\ge 1.
         \end{array}
         \right.
       \ee
\end{lem}
\begin{proof}
1. Let $L\ge 2$. By enlarging $K$ slightly, we may assume that the construction is carried out on a compact set $K_1$ containing $K$ in its interior. We continue to write $K$ for this enlarged compact set.
 Fix once and for all a universal constant $0<\epsilon<1/2$,
 and let $0<\delta<1$ be chosen later such that  $|\ln \rho|<\rho^{-\epsilon}$ for all $0<\rho\le \delta$.
Throughout the proof, $C>0$ denotes a constant depending only on $K$ and $L$ unless otherwise specified, which may vary from line to line. The dependence of constants on the fixed $\epsilon$ will be suppressed.
 By the form of $(u_0, p_0)$ in (\ref{eq3_u_0}) on $\Gamma_{\delta}$,  and using the homogeneous extension together with (\ref{eqE_0_1}) repeatedly to estimate the $\partial_3$-derivatives,
 there is a constant $C_0>0$, depending only on $L$ and $K$, such that for $\lambda\in K$, $k, l\in\mathbb{N}$ with $k+l\le L$, $|\beta|\le L$,
  and $x\in \Gamma_{\delta}$,
  \begin{equation}\label{eq4_1_0}
    \begin{split}
     & |\partial_3^l\partial_{\lambda}^{\beta}u_0'|\le C_0\rho(|\ln\rho|+1),  \ |\partial_{\rho}\partial_3^l\partial_{\lambda}^{\beta}u_0'|\le C_0(|\ln\rho|+1),  \ |\partial_{\rho}^k\partial_3^l\partial_{\lambda}^{\beta}u_0'|\le C_0\rho^{1-k},\textrm{ for }k\ge 2,\\
     & |\partial_3^l\partial_{\lambda}^{\beta}u_0^3|+|\partial_3^l\partial_{\lambda}^{\beta}p_0|\le C_0(|\ln\rho|+1), \ |\partial_{\rho}^k\partial_3^l\partial_{\lambda}^{\beta}u_0^3|+ |\partial_{\rho}^k\partial_3^l\partial_{\lambda}^{\beta}p_0|\le C_0\rho^{-k}, \textrm{ for }k\ge 1.
     \end{split}
  \end{equation}

\noindent 2.  \textbf{Claim}: There exists a constant
  $M_0>0$,  depending only on   $K$ and $L$,
  such that for all $n\ge 1$,
  $0\le k\le L$, $|\beta|\le L$,
  and  $x\in \Gamma_{\delta}$,
   \begin{equation}\label{eq4_M}
   \begin{split}
    & |\partial_{\rho}^k\partial_{\lambda}^{\beta}u_{n}'|\le n^{k}M_0^n\rho^{(2n+1)(1-\epsilon)-k}, \\
     & |\partial_{\rho}^k\partial_{\lambda}^{\beta}u_{n}^3|\le n^{k}M_0^n\rho^{2n(1-\epsilon)-k},\\
      & |\partial_{\rho}^k\partial_{\lambda}^{\beta}p_{n}|\le n^{k+1}M_0^n\rho^{2n(1-\epsilon)-k}.
    \end{split}
  \end{equation}
  \emph{Proof of Claim}:
  Choose an integer $N\ge L\ge 2$ sufficiently large such that $2n(1-\epsilon)-1>n$ for all $n\ge N$.   By the recursive construction of $(u_{n}, p_{n})$, all coefficients of the polynomials $a_n^j(t)$ and $c_n(t)$ depend polynomially on $\lambda$.  By this, the form of $(u_{n}, p_{n})$ in (\ref{eq3_u}), the degree bound (\ref{eq1_A2}) for the coefficients, the choice of $\delta$ that $|\ln\rho|<\rho^{-\epsilon}$ for $0<\rho<\delta$, and the compactness of $K$, there exists a constant $\bar{M}_0>1$, depending only on $K$ and $L$,   such that (\ref{eq4_M}) holds on $\Gamma_{\delta}$  for $1\le n\le N$, with $M_0$ replaced by $\bar{M}_0$.   Let $n\ge N$, and   let $M_0>\bar{M}_0$ be determined later.
   Assume (\ref{eq4_M}) holds for all  $1\le m\le n$. We show it also holds for $n+1$.

 Recall we denote $W_n=\sum_{m=0}^nW_{nm}$, where $W_{nm}=u_m\cdot \nabla u_{n-m}$. For $|\beta|\le L$,  applying  $\partial_{\lambda}^{\beta}$ to (\ref{eq1_3}), we obtain
  \be\label{eq4_M_E}
    \left\{
       \begin{split}
        & -\Delta' \partial_{\lambda}^{\beta}p_{n+1}=\partial_{\lambda}^{\beta}f_n^p=\partial_3^2 \partial_{\lambda}^{\beta}p_n+\partial_{\lambda}^{\beta} \dive W_n, \\
         & \Delta' \partial_{\lambda}^{\beta}u'_{n+1}=\partial_{\lambda}^{\beta}f_n'=-\partial_3^2\partial_{\lambda}^{\beta}u_n'+\partial_{\lambda}^{\beta}W_n'
         +\nabla_{x'} \partial_{\lambda}^{\beta}p_{n+1},\\
         & \Delta' \partial_{\lambda}^{\beta}u^3_{n+1}=\partial_{\lambda}^{\beta}f_n^3=-\partial^2_3\partial_{\lambda}^{\beta}u^3_n+\partial_{\lambda}^{\beta}W_n^3+\partial_3\partial_{\lambda}^{\beta}p_n,
       \end{split}
    \right.
 \ee
 on $\Gamma_{\delta}$. 
   We will estimate $\partial_{\rho}^k\partial_{\lambda}^{\beta}f_{n}^p, \partial_{\rho}^k\partial_{\lambda}^{\beta}f_{n}', \partial_{\rho}^k\partial_{\lambda}^{\beta}f_{n}^3$ for $0\le k\le L-2$
    using the induction hypothesis (\ref{eq4_M}), and then apply Lemma \ref{lem2_1} and Lemma \ref{lem2_2} to obtain the     corresponding estimates of derivatives of $(u_{n+1}, p_{n+1})$ up to order $L$.
Note $u_{n}$ is $(-1)$-homogeneous and $p_{n}$ is $(-2)$-homogeneous for any $n$.
By (\ref{eqE_0_4}) and the induction hypothesis (\ref{eq4_M}), we have,  for $1\le m\le n$, $k, l\ge 0$ satisfying $k+l\le L$, $|\beta|\le L$ and $x\in \Gamma_{\delta}$, that
\begin{equation}\label{eq4_M_1}
\begin{split}
& |\partial_{\rho}^k\partial_3^l\partial_{\lambda}^{\beta}u_m'|\le Cm^{k+l}M_0^{m}\rho^{(2m+1)(1-\epsilon)-k},\\
&   |\partial_{\rho}^k\partial_3^l\partial_{\lambda}^{\beta}u_m^3|\le Cm^{k+l}M_0^{m}\rho^{2m(1-\epsilon)-k},\\
&  |\partial_{\rho}^k\partial_3^l\partial_{\lambda}^{\beta}p_{m}|\le Cm^{k+l+1}M_0^{m}\rho^{2m(1-\epsilon)-k}.
\end{split}
\end{equation}
Below we use the index numbers $k_1, k_2, k_3\in\mathbb{N}$ and $\beta_1, \beta_2\in \mathbb{N}^4$.

\medskip

\noindent Step 2-1. We first estimate the nonlinear term $W_n$ and its derivatives. In particular, we show that for $0\le k\le L-1$, $|\beta|\le L$ and $x\in\Gamma_{\delta}$,
\be\label{eq4_W}
\begin{split}
&  |\partial_{\rho}^k\partial_{\lambda}^{\beta}W_{n}^j|  \le \sum_{m=0}^{n} |\partial_{\rho}^k\partial_{\lambda}^{\beta}W_{nm}^j|
      \le Cn^{k+2}M_0^{n}\rho^{(2n+2)(1-\epsilon)-1-k}, \quad j=\rho, \phi,\\
    & |\partial_{\rho}^k\partial_{\lambda}^{\beta}W_{n}^3|\le \sum_{m=0}^n|\partial_{\rho}^k\partial_{\lambda}^{\beta}W_{nm}^3|\le Cn^{k+2}M_0^{n}\rho^{(2n+1)(1-\epsilon)-1-k}.
    \end{split}
\ee
We only prove the estimates for $W_n^{\rho}$. The proofs for $W_n^{\phi}$ and $W_n^3$ are similar. By (\ref{eqE_6}), we have
  \[
    W_n^{\rho}=\sum_{m=0}^nW_{nm}^{\rho}, \quad \textrm{ where } W_{nm}^{\rho}=u_m^{\rho}\partial_{\rho}u_{n-m}^{\rho}+u_m^3\partial_3u_{n-m}^{\rho}-\rho^{-1}u_m^{\phi}u_{n-m}^{\phi}.
   \]
    Leibniz's rule gives
    \[
    \begin{split}
         |\partial_{\rho}^k\partial_{\lambda}^{\beta}W_{nm}^{\rho}|  & \le C\sum_{\substack{k_1+k_2=k\\ \beta_1+\beta_2=\beta}}(|\partial_{\rho}^{k_1} \partial_{\lambda}^{\beta_1}u_{m}^{\rho}||\partial_{\rho}^{k_2+1} \partial_{\lambda}^{\beta_2}u_{n-m}^{\rho}|+|\partial_{\rho}^{k_1} \partial_{\lambda}^{\beta_1}u_{m}^3||\partial_{\rho}^{k_2} \partial_3\partial_{\lambda}^{\beta_2}u_{n-m}^{\rho}|)\\
         & +C\sum_{\substack{k_1+k_2+k_3=k\\ \beta_1+\beta_2=\beta}}
         \rho^{-k_3-1}|\partial_{\rho}^{k_1}\partial_{\lambda}^{\beta_1}u_m^{\phi}||\partial_{\rho}^{k_2}\partial_{\lambda}^{\beta_2}u_{n-m}^{\phi}|
         \end{split}
    \]
 Since we take $k\le L-1$, the order of each $(\rho, x_3)$-derivative above is at most $L$, and $|\beta_j|\le |\beta|\le L$, $j=1, 2$. Note $|\ln\rho|<\rho^{-\epsilon}$ for $0<\rho<\delta$.
   Using (\ref{eq4_M_1}), and (\ref{eq4_1_0}) when $m=0$ or $n$, we have, for $x\in\Gamma_{\delta}$ and $0\le m\le n$,
   \[
   \begin{split}
      |\partial_{\rho}^k\partial_{\lambda}^{\beta}W_{nm}^{\rho}| & \le CM_0^n\rho^{(2n+2)(1-\epsilon)-1-k}\sum_{k_1=0}^k (m+1)^{k_1}(n-m+1)^{k-k_1+1}\\
      & \le Cn^{k+1}M_0^n\rho^{(2n+2)(1-\epsilon)-1-k}.
      \end{split}
         \]
   Then
   \[
   \begin{split}
      |\partial_{\rho}^k\partial_{\lambda}^{\beta}W_{n}^{\rho}| & \le \sum_{m=0}^n|\partial_{\rho}^k\partial_{\lambda}^{\beta}W_{nm}^{\rho}|
      \le Cn^{k+2}M_0^{n}\rho^{(2n+2)(1-\epsilon)-1-k}.
     \end{split}
   \]

   By similar computations, using (\ref{eqE_6}), (\ref{eq4_1_0}) and (\ref{eq4_M_1}), we also have the estimates of $\partial_{\rho}^k\partial_{\lambda}^{\beta}W_{n}^{j}$ for $j=\phi, 3$ in (\ref{eq4_W}).

 \medskip

   \noindent Step 2-2. We now estimate $\partial_{\rho}^k\partial_{\lambda}^{\beta}f_{n}^{p}$ and
   $\partial_{\rho}^k\partial_{\lambda}^{\beta}p_{n+1}$.

   Note $W_n$ is $(-3)$-homogeneous. By (\ref{eq4_W}) and (\ref{eqE_0_1}), we have, for $0\le k\le L-2$ and $x\in\Gamma_{\delta}$, that
   \[
   \begin{split}
     |\partial_{\rho}^k\partial_{\lambda}^{\beta}\dive W_n| & \le
     |\partial_{\rho}^k\partial_{\lambda}^{\beta}(\rho^{-1}\partial_{\rho}(\rho W_n^{\rho})-\rho^{-2}\partial_{\rho}(\rho^3W_n^3))|\\
       & \le \sum_{k_1=0}^{k+1}(\rho^{-k-1+k_1}|\partial_{\rho}^{k_1}\partial_{\lambda}^{\beta}W_n^{\rho}|+\rho^{-k+k_1}|\partial_{\rho}^{k_1}\partial_{\lambda}^{\beta}W_n^{3}|)\\
     & \le Cn^{k+3}M_0^{n}\rho^{2(n+1)(1-\epsilon)-2-k}.
     \end{split}
   \]
For $0\le k\le L-2$, the order of $\partial_{\rho}^k\partial_3^2\partial_{\lambda}^{\beta}p_n$
 satisfies $k+2\le L$ and $|\beta|\le L$. Combining the above with (\ref{eq4_M_1}),  we have
   \[
    |\partial_{\rho}^k\partial_{\lambda}^{\beta}f_{n}^p|\le |\partial_{\rho}^k\partial_3^2\partial_{\lambda}^{\beta}p_n|+|\partial_{\rho}^k\partial_{\lambda}^{\beta}\dive W_n|\le Cn^{k+3}M_0^{n}\rho^{2(n+1)(1-\epsilon)-2-k}, \ x\in\Gamma_{\delta}.
  \]
  By the construction in  Lemma \ref{lem3_2},
  $p_{n+1}=\rho^{2n+2}c_{n+1}(t)$ on $\Gamma_{\delta}$, where $c_{n+1}(t)$ is a polynomial. For $n\ge 0$, we have $\lim_{\rho\to 0^+}\partial_{\lambda}^{\beta}p_{n+1}=0$.
        Apply Lemma \ref{lem2_1} with $\alpha=2(n+1)(1-\epsilon)-2$ and $0\le k\le L-2$ to the first equation in (\ref{eq4_M_E}).
        Since $0<\epsilon<1/2$ and we chose $N$ such that $2n(1-\epsilon)-1>n$ for $n\ge N$, we have
   \[
     n<|\alpha|+1=2n(1-\epsilon)+1-2\epsilon\le 4n, \quad n\ge N.
   \]
        The quantitative estimate (\ref{eq2_1_2}) gives
    \begin{equation}\label{eq4_1_11}
   \begin{split}
    & |\partial_{\rho}^{k}\partial_{\lambda}^{\beta}p_{n+1}| \le    C_pn^{k+1}M_0^{n}\rho^{2(n+1)(1-\epsilon)-k},  \quad  0\le k\le L, x\in \Gamma_{\delta},
           \end{split}
   \end{equation}
   for some constant $C_p>0$ depending only on $K$ and $L$.

   \medskip

   \noindent Step 2-3.  Next, we estimate $\partial_{\rho}^k\partial_{\lambda}^{\beta}f_{n}'$ and
   $\partial_{\rho}^k\partial_{\lambda}^{\beta}u'_{n+1}$.

   By  (\ref{eq4_M_1}), (\ref{eq4_W}), and (\ref{eq4_1_11}),    we have,  for $j=\rho, \phi$, $0\le k\le L-2$, and $x\in\Gamma_{\delta}$, that
\[
      |\partial_{\rho}^{k}\partial_{\lambda}^{\beta}f_{n}^{j}|\le |\partial_{\rho}^{k}\partial_3^2\partial_{\lambda}^{\beta}u_{n}^j|+|\partial_{\rho}^{k}\partial_{\lambda}^{\beta}W_{n}^j|+|\partial_{\rho}^{k+1}\partial_{\lambda}^{\beta}p_{n+1}|\le Cn^{k+2}M_0^{n}\rho^{(2n+2)(1-\epsilon)-1-k}.
   \]
   Indeed, the pressure term above is present only when $j=\rho$.
   By the construction in  Lemma \ref{lem3_2}, $u_{n+1}^j=\rho^{2n+3}a_{n+1}^j(t)$, where $a_{n+1}^j(t)$ are polynomials, $j=\rho, \phi$. For $n\ge 0$, we have $\lim_{\rho\to 0^+}\partial_{\lambda}^{\beta}u_{n+1}^j=0$ and $\lim_{\rho\to 0^+}\partial_{\rho}\partial_{\lambda}^{\beta}u_{n+1}^j=0$, $j=\rho, \phi$.        Apply Lemma \ref{lem2_2} with $\alpha=(2n+2)(1-\epsilon)-1$ and $0\le k\le L-2$  to the second equation in (\ref{eq4_M_E}). Arguing as in step 2-2, we have $n<|\alpha|+1\le 4n$ for $n\ge N$.
     Then the quantitative estimate (\ref{eq2_2_4}) gives, for $0\le k\le L$, that
   \begin{equation}\label{eq4_1_14}
   |\partial_{\rho}^{k}\partial_{\lambda}^{\beta}u'_{n+1}|\le C'n^{k}M_0^{n}\rho^{(2n+2)(1-\epsilon)+1-k}\le C'n^{k}M_0^{n}\rho^{(2n+3)(1-\epsilon)-k},\  x\in\Gamma_{\delta},
   \end{equation}
    for some $C'>0$ depending only on $K$ and $L$.

    \medskip

\noindent Step 2-4. Now we estimate $\partial_{\rho}^k\partial_{\lambda}^{\beta}f_{n}^3$ and
  $\partial_{\rho}^k\partial_{\lambda}^{\beta}u_{n+1}^3$.
   By  (\ref{eq4_M_1}) and the second line in (\ref{eq4_W}), we have, for $0\le k\le L-2$ and $x\in\Gamma_{\delta}$, that
      \[
     |\partial_{\rho}^k\partial_{\lambda}^{\beta}f_{n}^3|\le |\partial_{\rho}^k\partial_3^2\partial_{\lambda}^{\beta}u_{n}^3|+|\partial_{\rho}^k\partial_{\lambda}^{\beta}W_{n}^3|+|\partial_{\rho}^k\partial_3\partial_{\lambda}^{\beta}p_{n}|\le  Cn^{k+2}M_0^{n}\rho^{(2n+1)(1-\epsilon)-1-k}.
        \]
   By the construction in  Lemma \ref{lem3_2}, $u^3_{n+1}$ has the form $u_{n+1}^3=\rho^{2n+2}a_{n+1}^3(t)$ on $\Gamma_{\delta}$, where $a_{n+1}^3(t)$ is a polynomial.
   Hence $\lim_{\rho\to 0^+}\partial_{\lambda}^{\beta}u_{n+1}^3=0$.
   Apply Lemma \ref{lem2_1} with $\alpha=(2n+1)(1-\epsilon)-1$ and $0\le k\le L-2$ to the third equation in (\ref{eq4_M_E}). Arguing as in step 2-2, we have $n<|\alpha|+1\le 4n$ for $n\ge N$.
   Then (\ref{eq2_1_2}) gives, for $0\le k\le L$, that
       \begin{equation}\label{eq4_1_9}
    |\partial_{\rho}^{k}\partial_{\lambda}^{\beta}u_{n+1}^3|
   \le C_3n^{k}M_0^{n}\rho^{(2n+1)(1-\epsilon)+1-k}
   \le C_3n^{k}M_0^{n}\rho^{2(n+1)(1-\epsilon)-k}, \
    x\in\Gamma_{\delta},
   \end{equation}
   for some constant $C_3>0$ depending only on  $K$ and $L$.

 Choose $M_0>\max\{\bar{M}_0, C_3, C_p, C'\}$. Combining (\ref{eq4_1_11}), (\ref{eq4_1_14}) and (\ref{eq4_1_9}), we obtain (\ref{eq4_M}) for $(u_{n+1}, p_{n+1})$.
   The Claim is proved.

\medskip

\noindent 3. Let $L\ge 2$ be fixed.
We show that there exists some $\delta_L>0$, depending only on $L$ and $K$, such that for any compact set $V\subset\subset \Omega_{\delta_L}\setminus\{x'=0\}$, the series $\sum_{n=0}^{\infty}u_n$ and $\sum_{n=0}^{\infty}p_n$ converge in $C^L(K\times V)$.
Since $u_n$ is $(-1)$-homogeneous and $p_n$ is $(-2)$-homogeneous,  it follows from  (\ref{eqE_0_4}) and (\ref{eq4_M}) that, for $n\ge 1$, every multi-index $\beta\in\mathbb{N}^4$  and $k, l\in\mathbb{N}$ satisfying %$|\beta|\le L$ and $k+l\le L$, %
$|\beta|+k+l\le L$, 
we have
\be\label{eq4_1_17}
   \begin{split}
& |\nabla_{x'}^k\partial_3^l\partial_{\lambda}^{\beta}u_n'|\le Cn^{k+l}M_0^{n}\rho^{(2n+1)(1-\epsilon)-k},\\
&   |\nabla_{x'}^k\partial_3^l\partial_{\lambda}^{\beta}u_n^3|\le Cn^{k+l}M_0^{n}\rho^{2n(1-\epsilon)-k},\\
&  |\nabla_{x'}^k\partial_3^l\partial_{\lambda}^{\beta}p_{n}|\le Cn^{k+l+1}M_0^{n}\rho^{2n(1-\epsilon)-k},
\end{split}
\ee
where $C>0$ depends only on  $K$ and $L$.
 Choose $0<\delta_L<\delta$ sufficiently small such that
 \[
    M_0\delta_L^{2(1-\epsilon)}<1.
 \]
  For each $0\le k+l\le L$, the powers of $\rho$ in (\ref{eq4_1_17}) are positive for all sufficiently large $n$. Thus the series $\sum_{n=0}^{\infty}\partial_{\lambda}^{\beta}\nabla_{x'}^k\partial_3^lu_{n}$ and $\sum_{n=0}^{\infty}\partial_{\lambda}^{\beta}\nabla_{x'}^k\partial_3^lp_{n}$ are absolutely and uniformly convergent on $K\times \Gamma_{\delta_L}$, for all $\beta\in\mathbb{N}^4$  and $k, l\in\mathbb{N}$ satisfying %$|\beta|\le L$ and $k+l\le L$. 
  $|\beta|+k+l\le L$.
By homogeneity of $(u_n, p_n)$, the convergence is locally uniform on $K\times(\Omega_{\delta_L}\setminus\{x'=0\})$.
Since $\epsilon$ has been fixed throughout the proof, the constants $M_0$ and  $\delta_L$ depend only on $K$ and $L$.

\medskip

\noindent 4. Now we prove part (i) of the lemma. Take $L=2$ and set $\delta=\delta_2$. By the argument in Step 3 with $\beta=0$, the series $u_{\lambda}:=\sum_{n=0}^{\infty}u_{n}$ and $p_{\lambda}:=\sum_{n=0}^{\infty}p_{n}$
   are well-defined and  converge in $C^2(V)$ for any compact subset $V\subset \Omega_{\delta}\setminus\{x'=0\}$, uniformly for $\lambda\in K$.
   By the axisymmetry and homogeneity of $u_n, p_n$, we have that $u_{\lambda}, p_{\lambda}$ are axisymmetric and homogeneous of degrees $-1$ and $-2$ respectively. By the Claim and  the explicit form of the leading term $(u_0, p_0)$, we have that $(u_{\lambda}, p_{\lambda})$ satisfies (\ref{eq1_a_3}) and $|u_{\lambda}|=o(|x'|^{-1})$ as $|x'|\to 0$ for each fixed $x_3>0$.

Next, we show $(u_{\lambda}, p_{\lambda})$ is a solution of (\ref{NS}) in $\Omega_{\delta}\setminus\{x'=0\}$. By the definition of $(u_0, p_0)$,  we have
\be\label{eq4_1_15}
-\Delta' p_0=0,\quad
     \Delta'u_0'=\nabla_{x'}p_0,  \quad \Delta'u_0^3=0.
\ee
By the construction in Lemma \ref{lem3_2}, after homogeneous extension,
the recursive system (\ref{eq1_3}) holds in $\Omega_{\delta}\setminus\{x'=0\}$ for all $n\ge 0$.
By (\ref{eq4_1_0}) and (\ref{eq4_M}), we see that  (\ref{eq4_W}) holds for every $n\ge 1$, while the corresponding estimates  for $n=0$ hold with $n$ replaced by $1$ on the right-hand sides of the estimates in (\ref{eq4_W}).
Taking $k=0$ and $\beta=0$, we obtain
\[
\sum_{n=0}^{\infty}\sum_{m=0}^{n}|u_m\cdot\nabla u_{n-m}|=\sum_{n=0}^{\infty}\sum_{m=0}^{n}|W_{nm}|\le C\sum_{n=0}^{\infty}(n^2+1)M_0^n\rho^{(2n+1)(1-\epsilon)-1}<\infty, \quad \textrm{ on }\Gamma_{\delta},
\]
and the series is absolutely and locally uniformly convergent in $K\times (\Omega_{\delta}\setminus\{x'=0\})$. So
\[
  \sum_{n=0}^{\infty}\sum_{m=0}^{n}u_m\cdot\nabla u_{n-m}=\sum_{m=0}^{\infty}u_m\cdot \sum_{i=0}^{\infty}\nabla u_{i}=u_{\lambda}\cdot \nabla u_{\lambda}.
\]
Summing the last two equations in (\ref{eq4_1_15}) together with the last two equations in (\ref{eq1_3}) over all $n\ge 0$,   we obtain
\[
  \Delta' u_{\lambda}=-\partial_3^2u_{\lambda}+u_{\lambda}\cdot\nabla u_{\lambda}+\nabla p_{\lambda}, \quad \textrm{ in }\Omega_{\delta}\setminus\{x'=0\}.
\]
By Lemma \ref{lem3_2}, we have $\dive u_n=0$ for all $n\ge 0$.
Passing to the limit in the locally uniformly convergent series $\sum_{n=0}^{\infty}\dive u_n$ gives $\dive u_{\lambda}=0$.
Therefore,
$(u_{\lambda}, p_{\lambda})\in C^2(\Omega_{\delta}\setminus\{x'=0\})$ is a solution of (\ref{NS}) in $\Omega_{\delta}\setminus\{x'=0\}$. Write the above equation as
\[
  -\Delta u_{\lambda}+\nabla p_{\lambda}=-u_{\lambda}\cdot \nabla u_{\lambda}.
\]
By a bootstrap argument using standard interior estimates for Stokes equations away from $\{x'=0\}$,
  we obtain $(u_{\lambda}, p_{\lambda})\in C^{\infty}(\Omega_{\delta}\setminus\{x'=0\})$. So (i) is proved.

\medskip

 \noindent 5. Now we prove (ii). Fix $L\ge 2$ and $K$, let $(u_{\lambda}, p_{\lambda})$ be the solution of (\ref{NS}) in $\Omega_{\delta}\setminus\{x'=0\}$ constructed above. By the Claim and the argument in Step 3,  there exists some $0<\delta'\le \delta$, depending only on $L$ and $K$,  such that the series $\sum_{n=0}^{\infty}\partial_{\rho}^k\partial_{\lambda}^{\beta}u_{n}$ and $\sum_{n=0}^{\infty}\partial_{\rho}^k\partial_{\lambda}^{\beta}p_{n}$
 are absolutely and locally uniformly convergent on $K\times \Gamma_{\delta'}$,
 for all $|\beta|\le L$ and $0\le k\le L$. 
 Consequently, (\ref{eqprop4_3}) and (\ref{eqprop4_2}) hold for $0<\rho<\delta'$, $0\le k\le L$ and $|\beta|\le L$. Moreover, $(u_{\lambda}, p_{\lambda})$ depends smoothly on $\lambda$ on $\Gamma_{\delta'}$.

 It remains to show that the same estimates hold on $\Gamma_{\delta}\cap \{\delta'\le \rho\le\delta\}$. Note the reduced system (\ref{eqNC_C}) is an ODE system in $\rho$ with smooth coefficients and every $\rho>0$ is an ordinary point. By the above arguments,  for any $|\beta|\le L$ and $0\le k\le L$,
 $\partial_{\rho}^k\partial_{\lambda}^{\beta}u_{\lambda}|_{\rho=\delta'/2}$ and $\partial_{\rho}^k\partial_{\lambda}^{\beta}p_{\lambda}|_{\rho=\delta'/2}$ are uniformly bounded for $\lambda\in K$.
  Standard ODE theory
  implies that $(u_{\lambda}, p_{\lambda})$ is smooth in
   $(\lambda, x)\in K\times \{\delta'/2\le \rho\le\delta, x_3=1\}$, and
  \[
    |\partial_{\rho}^k\partial_{\lambda}^{\beta}u_{\lambda}|+ |\partial_{\rho}^k\partial_{\lambda}^{\beta}p_{\lambda}|\le C, \quad \delta'/2\le \rho< \delta, \ x_3=1, \ 0\le k\le L,\  |\beta|\le L,
  \]
where $C>0$ depends only on $L, \delta, \delta'$ and $K$.
 Since $\delta$ depends only on $K$ and $\delta'$ depends only on $L$ and $K$, we have that $C$ depends only on $L$ and $K$. Since the weights in  (\ref{eqprop4_3}) and (\ref{eqprop4_2}) are bounded below by positive constants on $\delta'\le \rho\le \delta$, the estimates also hold there after increasing $C$.
 Thus  (\ref{eqprop4_3}), (\ref{eqprop4_2}) hold for all $0<\rho<\delta$.  (ii) is proved.
\end{proof}

Next, we show that every local $(-1)$-homogeneous axisymmetric solution $(u, p)$ of (\ref{NS}) near the positive $x_3$-axis satisfying $|u|=o(|x'|^{-1})$ as $x'\to 0$ for each fixed $x_3>0$ is  given by the series solutions we constructed above.
%
%Next, we show that all local Type II $(-1)$-homogeneous axisymmetric solutions of (\ref{NS}) near the positive $x_3$-axis are given by the series solutions we constructed above.
%
% Consider such a  solutions $(u, p)$ of (\ref{NS}) in $\Omega_{\delta}\setminus\{x'=0\}$ for some $\delta>0$. %Consider a local Type II $(-1)$-homogeneous axisymmetric solution $(u, p)$ of (\ref{NS}) in $\Omega_{\delta}\setminus\{x'=0\}$ for some $\delta>0$.
 By Lemma \ref{lem3_1}, there exists a unique $\lambda=(c_0, d_0, b_{0}^{\rho}, b_{0}^{\phi})$ associated with $(u, p)$ through  (\ref{eq1_a_3}). Let $\{(u_n, p_n)\}_{n=0}^{\infty}$ be the sequence constructed in Lemma \ref{lem3_2} with this $\lambda$.
For each $N\ge 0$, set
\be\label{eq4_M_2_N}
    U_{N}=\sum_{n=0}^{N}u_{n}, \quad h_{N}=u-U_{N}, \quad P_{N}=\sum_{n=0}^{N}p_{n}, \quad r_{N}=p-P_{N}.
\ee
 \begin{lem}\label{lem4_M_2}
  Let $\delta>0$ and $(u, p)\in C^2(\Omega_{\delta}\setminus\{x'=0\})$ be a $(-1)$-homogeneous axisymmetric solution of (\ref{NS}) in $\Omega_{\delta}\setminus\{x'=0\}$
  satisfying
  \[
  |u|=o(|x'|^{-1}), \textrm{ as }x'\to 0
  \]
   for each fixed $x_3>0$. Let $\lambda=(c_0, d_0, b_{0}^{\rho}, b_{0}^{\phi})$ be associated with $(u, p)$ through (\ref{eq1_a_3}), let $\{(u_n, p_n)\}_{n=0}^{\infty}$ be the sequence constructed in Lemma \ref{lem3_2}, and let $h_N, r_N$ be defined by (\ref{eq4_M_2_N}) for $N\ge 0$.
  Then for any $\epsilon>0$, there exist constants $0<\delta'\le \delta$ and $K_0>0$, depending only on $(u, p)$,   $\delta$ and $\epsilon$,
  such that for every $N\ge 0$, $0\le k\le 2$ and $x\in \Gamma_{\delta'}$,
   \begin{equation}\label{eq4_M_h}
   \begin{split}
  &  |\nabla_{x'}^kh_{N}'|\le (N+1)^{k}K_0^{N+1}\rho^{(2N+3)(1-\epsilon)-k},\\
  &  |\nabla_{x'}^kh_{N}^3|\le (N+1)^{k}K_0^{N+1}\rho^{2(N+1)(1-\epsilon)-k}, \\
  &  |\nabla_{x'}^kr_{N}|\le (N+1)^{k+1}K_0^{N+1}\rho^{2(N+1)(1-\epsilon)-k}.
   \end{split}
  \end{equation}
  \end{lem}
  \begin{proof}
1. It suffices to prove the result for $0<\epsilon<1/2$, since the estimate for a larger $\epsilon$ follows from the estimate for a smaller $\epsilon$ after decreasing $\delta'<1$ if necessary.
By the Claim in the proof of Lemma \ref{lem4_M_1} with $L=2$, $K=\{\lambda\}$ and the present $\epsilon$, together with (\ref{eqE_0_4}), there exist constants $M_0$ and $0<\bar{\delta}<\min\{\delta, 1\}$, depending only on $(u, p)$, $\delta$,
and $\epsilon$, such that
(\ref{eq4_M_1}) holds for all $k, l\in\mathbb{N}$ satisfying $k+l\le 2$, $m\ge 1$, $|\beta|=0$ and $x\in\Gamma_{\bar{\delta}}$. %$0<\rho<\bar{\delta}$.

By Lemma \ref{lem3_1}, there exists some $\bar{K}_0>0$,  depending only on $(u, p)$, $\delta$ and $\epsilon$,  such that (\ref{eq4_M_h}) holds for $N=0$ and  $0<\rho<\bar{\delta}$, with $K_0$ replaced by $\bar{K}_0$.
  Throughout the proof, $C>0$ denotes a constant depending only on $(u, p)$, $\delta$ and $\epsilon$,
  which may vary from line to line.

  Assume that (\ref{eq4_M_h})  holds for $(h_n, r_n)$ for all $0\le n\le N$.
  We show it also holds for $N+1$.
  Write $u=U_{N}+h_{N}$, $p=P_{N}+r_{N}$, and
  \[
  \begin{split}
    W & =u\cdot \nabla u=(U_{N}+h_{N})\cdot \nabla (U_{N}+h_{N})=\sum_{n=0}^NW_n+I_N,
\end{split}
  \]
 where $I_N=I_{N1}+I_{N2}+I_{N3}$, with
\[
  I_{N1}=\sum_{\substack{i+m\ge  N+1\\ 0\le i, m\le N}}u_{m}\cdot \nabla u_{i},
    \quad I_{N2}=U_{N}\cdot \nabla h_{N}+h_{N}\cdot \nabla U_{N}, \quad I_{N3}=h_{N}\cdot \nabla h_{N}.
\]
  By (\ref{NS}), we have
  \[
  \begin{split}
  & -\Delta'p=\partial_3^2p+\dive (u\cdot \nabla u)\\
   &  \Delta' u=-\partial_3^2u+u\cdot \nabla u+\nabla p.
    \end{split}
  \]
  Adding (\ref{eq4_1_15}) to the sum of (\ref{eq1_3}) over $0\le n\le N$ and subtracting the resulting equations from the equations above,
  we obtain
  \begin{equation}\label{eq4_2_1}
     \left\{
     \begin{split}
     & -\Delta' r_{N+1}=\partial_3^2r_{N}+\dive I_N=:g_{N}^{p}, \\
     &  \Delta' h_{N+1}'=-\partial_3^2h_{N}'+I'_N+\nabla' r_{N+1}=:g_{N}', \\
     & \Delta' h_{N+1}^3=-\partial_3^2h_{N}^3+I^3_N+\partial_{3}r_{N}=:g_{N}^3.
     \end{split}
     \right.
  \end{equation}
  We will estimate $g_{N}^p,  g_{N}', g_{N}^3$
  and then apply
  Lemma \ref{lem2_1} and Lemma \ref{lem2_2} to derive estimates of $h_{N+1}$ and $r_{N+1}$.
   Let $K_0>\max\{M_0, \bar{K}_0, 1\}$ be chosen later,     and then choose $0<\delta'<\min\{\bar{\delta}, 1\}$ sufficiently small such that
   \be\label{eq4_2_4}
     |\ln\rho|<\rho^{-\epsilon}, \quad M_0\rho^{2(1-\epsilon)}<\frac{1}{2}, \quad K_0\rho^{2(1-\epsilon)}<1, \quad \forall 0<\rho<\delta'.
     \ee
 By summing the estimates in (\ref{eq4_1_0}) and (\ref{eq4_M_1}) for $0\le n\le N$, in view of (\ref{eqE_0_4}), we have,  for $x\in\Gamma_{\delta'}$, that
\begin{equation}\label{eq4_2_U}
  \begin{split}
     & |\partial_3^lU_N'|\le C\rho(|\ln\rho|+1), \quad |\partial_3^lU_N^3|\le C(|\ln\rho|+1), \textrm{ for }0\le l\le 2, \\
     &  |\partial_{\rho}\partial_3^lU_N'|\le C(|\ln\rho|+1), \quad |\partial_{\rho}\partial_3^lU_N^3|\le C\rho^{-1}, \textrm{ for }l=0, 1,  \\
     & |\partial_{\rho}^2U_N'|\le C\rho^{-1},  \quad |\partial_{\rho}^2U_N^3|\le C\rho^{-2}.
     \end{split}
\end{equation}
Since $h_{N}$ is $(-1)$-homogeneous and $r_{N}$ is $(-2)$-homogeneous, it follows from (\ref{eqE_0_4}) and the induction hypothesis (\ref{eq4_M_h}) that,  for $0\le k+l\le 2$ and $N\ge 0$, we have
\begin{equation}\label{eq4_2_3}
\begin{split}
& |\partial_{\rho}^k\partial_3^lh_{N}'|\le C(N+1)^{k+l}K_0^{N+1}\rho^{(2N+3)(1-\epsilon)-k}, \\
&  |\partial_{\rho}^k\partial_3^lh_{N}^3|\le C(N+1)^{k+l}K_0^{N+1}\rho^{2(N+1)(1-\epsilon)-k},\\
&  |\partial_{\rho}^k\partial_3^lr_{N}|\le C(N+1)^{k+l+1}K_0^{N+1}\rho^{2(N+1)(1-\epsilon)-k}.
\end{split}
\end{equation}

  \noindent 2. We first estimate $I_N$ and show that for $x\in \Gamma_{\delta'}$ and $0\le k\le 1$,
  \be\label{eq4_I_N}
  \begin{split}
  & |\partial_{\rho}^kI^j_N|\le C(N+1)^{k+2}K_0^{N+1}\rho^{(2N+3)(1-\epsilon)-\epsilon-k},\ j=\rho, \phi, \\
    &   |\partial_{\rho}^kI_N^3|\le C(N+1)^{k+2}K_0^{N+1}\rho^{2(N+1)(1-\epsilon)-\epsilon-k}.
      \end{split}
  \ee
  We only prove it for $I_N^{\rho}$. The proofs for $I_N^{\phi}$ and $I_N^3$ are similar.
  Let $k_1, k_2, k_3\in\mathbb{N}$ denote index numbers below. By (\ref{eqE_6}) and  (\ref{eq4_M_1}) with $|\beta|=0$, using that $M_0\rho^{2(1-\epsilon)}<1/2$ and $K_0\ge M_0$,
  we have, for $x\in\Gamma_{\delta'}$ and $k=0, 1$, that
    \begin{equation*}
   \begin{split}
      |\partial_{\rho}^kI_{N1}^{\rho}| &
      \le \sum_{n=N+1}^{2N}\sum_{\substack{m+i=n, \\ 1\le m, i\le N}}(\sum_{k_1+k_2=k}(|\partial_{\rho}^{k_1}u_{m}^{\rho}\partial_{\rho}^{k_2+1}u_{i}^{\rho}|+|\partial_{\rho}^{k_1}u_{m}^3\partial_{\rho}^{k_2}\partial_3u_{i}^{\rho}|)\\
      & + \sum_{k_1+k_2+k_3=k}\rho^{-k_3-1}|\partial_{\rho}^{k_1}u_m^{\phi}\partial_{\rho}^{k_2}u_i^{\phi}|)\\
      &  \le C\sum_{n=N+1}^{2N}(\sum_{m=n-N}^{N}\sum_{k_1=0}^{k}m^{k_1}(n-m)^{k-k_1+1})M_0^{n}\rho^{(2n+1)(1-\epsilon)-\epsilon-k}\\
   & \le C(N+1)^{k+2}K_0^{N+1}\rho^{(2N+3)(1-\epsilon)-\epsilon-k}.
   \end{split}
   \end{equation*}
   By the induction hypothesis (\ref{eq4_M_h}), together with (\ref{eq4_2_U}), (\ref{eq4_2_3}) and (\ref{eq4_2_4}), we have, for $x\in\Gamma_{\delta'}$ and $k=0, 1$, that
    \[
   \begin{split}
      |\partial_{\rho}^{k}I_{N2}^{\rho}|  & \le \sum_{k_1+k_2=k}(|\partial_{\rho}^{k_1}U_{N}^{\rho}\partial_{\rho}^{k_2+1}h_{N}^{\rho}|+|\partial_{\rho}^{k_1}U_{N}^3\partial_{\rho}^{k_2}\partial_3h_{N}^{\rho}|+ |\partial_{\rho}^{k_1}h_{N}^{\rho}\partial_{\rho}^{k_2+1}U_{N}^{\rho}|+|\partial_{\rho}^{k_1}h_{N}^3\partial_{\rho}^{k_2}\partial_3U_{N}^{\rho}|) \\
           & +\sum_{k_1+k_2+k_3=k}\rho^{-k_3-1}|\partial_{\rho}^{k_1}U_N^{\phi}\partial_{\rho}^{k_2}h_N^{\phi}| \le C(N+1)^{k+1}K_0^{N+1}\rho^{(2N+3)(1-\epsilon)-\epsilon-k},
           \end{split}
           \]
           \[
           \begin{split}
      |\partial_{\rho}^{k}I_{N3}^{\rho}| &  \le \sum_{k_1+k_2=k}(|\partial_{\rho}^{k_1}h_{N}^{\rho}\partial_{\rho}^{k_2+1}h_{N}^{\rho}|+|\partial_{\rho}^{k_1}h_{N}^3\partial_{\rho}^{k_2}\partial_3h_{N}^{\rho}|) +\sum_{k_1+k_2+k_3=k}\rho^{-k_3-1}|\partial_{\rho}^{k_1}h_N^{\phi}\partial_{\rho}^{k_2}h_N^{\phi}|\\
      &  \le C(N+1)^{k+1}K_0^{2N+2}\rho^{(4N+5)(1-\epsilon)-\epsilon-k}
        \le  C(N+1)^{k+1}K_0^{N+1}\rho^{(2N+3)(1-\epsilon)-\epsilon-k}.
      \end{split}
  \]
   Combining  the above estimates, we obtain (\ref{eq4_I_N}) for $I_N^{\rho}$. The proof of (\ref{eq4_I_N}) for $j=\phi, 3$ is similar.

   \medskip

   \noindent 3. Now we estimate $r_{N+1}$.
  Note $I_N$ is $(-3)$-homogeneous. By (\ref{eqE_0_1}) and (\ref{eq4_I_N}), we obtain
  \[
  \begin{split}
    |\dive I_N| & =|\rho^{-1}\partial_{\rho}(\rho I_N^{\rho})-\rho^{-2}\partial_{\rho}(\rho^3I_N^3)|
     \le \rho^{-1}|I_N^{\rho}|+|\partial_{\rho}I_N^{\rho}|+3|I_N^3|+\rho|\partial_{\rho}I_N^3|\\
     & \le C(N+1)^3K_0^{N+1}\rho^{2(N+1)(1-\epsilon)-2\epsilon}.
    \end{split}
  \]
Using this
and (\ref{eq4_2_3}), we have
\[
   |g_{N}^p|\le |\partial_{3}^2r_{N}|+|\dive I_N|\le C(N+1)^3K_0^{N+1}\rho^{2(N+1)(1-\epsilon)-2\epsilon}, \quad \textrm{ in }\Gamma_{\delta'}.
\]
Since $r_{N+1}=r_N-p_{N+1}$ and $\lim_{\rho\to 0^+}r_{N}=\lim_{\rho\to 0^+}p_{N+1}=0$, we have  $\lim_{\rho\to 0^+}r_{N+1}=0$.
Apply Lemma \ref{lem2_1} with $L=0$ and $\alpha=2(N+1)(1-\epsilon)-2\epsilon$ to the first equation of (\ref{eq4_2_1}).
For the fixed $0<\epsilon<1/2$, there exists a constant  $C_{\epsilon}>0$, such that $(N+1)/C_{\epsilon}\le |\alpha|+1\le C_{\epsilon}(N+1)$ for all $N\ge 0$.
The quantitative estimate (\ref{eq2_1_2}),  together with (\ref{eq3_E_1}),  gives, for $0\le k\le 2$ and $x\in\Gamma_{\delta'}$, that
   \begin{equation}\label{eq4_2_16}
    \begin{split}
   |\nabla^k_{x'}r_{N+1}| & \le C_p(N+1)^{k+1}K_0^{N+1}\rho^{2(N+1)(1-\epsilon)-2\epsilon+2-k}
    = C_p(N+1)^{k+1}K_0^{N+1}\rho^{2(N+2)(1-\epsilon)-k},
   \end{split}
   \end{equation}
  where $C_p>0$ depends only on $(u, p), \delta$ and $\epsilon$.

\medskip

   \noindent 4. Next, we estimate $h_{N+1}'$.
   By (\ref{eq4_M_h}), (\ref{eq4_2_3}), (\ref{eq4_I_N}) and (\ref{eq4_2_16}), we have
   \[
     |g_{N}'|\le |\partial_3^2h_{N}'|+|I'_N|+|\nabla_{x'}r_{N+1}|\le C(N+1)^2K_0^{N+1}\rho^{(2N+3)(1-\epsilon)-\epsilon}, \quad \textrm{ in }\Gamma_{\delta'}.
   \]
  Since $h_{N+1}'=h_{N}'-u_{N+1}'$, $\lim_{\rho\to 0^+}h_{N}'=\lim_{\rho\to 0^+}u_{N+1}'=0$ and $\lim_{\rho\to 0^+}\partial_{\rho}h_{N}'=\lim_{\rho\to 0^+}\partial_{\rho}u_{N+1}'=0$, we have $\lim_{\rho\to 0^+}h_{N+1}'=0$ and $\lim_{\rho\to 0^+}\partial_{\rho}h_{N+1}'=0$. Apply Lemma \ref{lem2_2} with $L=0$ and $\alpha=(2N+3)(1-\epsilon)-\epsilon$ to the second equation in  (\ref{eq4_2_1}). As in step 3, we also have $(N+1)/C_{\epsilon}\le |\alpha|+1\le C_{\epsilon}(N+1)$ for all $N\ge 0$, for some $C_{\epsilon}>0$ depending only on $\epsilon$.
  The quantitative estimate (\ref{eq2_2_4}),  together with (\ref{eq3_E_1}),  gives, for $0\le k\le 2$ and $x\in \Gamma_{\delta'}$, that
   \begin{equation}\label{eq4_2_23}
   \begin{split}
    |\nabla^k_{x'}h_{N+1}'| & \le C'(N+1)^kK_0^{N+1}\rho^{(2N+3)(1-\epsilon)-\epsilon+2-k}
     \le C'(N+1)^kK_0^{N+1}\rho^{(2N+5)(1-\epsilon)-k},
    \end{split}
   \end{equation}
   for some constant $C'>0$ depending only on $(u, p), \delta$ and $\epsilon$.

\medskip

      \noindent 5. Now we estimate $h_{N+1}^3$.

 By  (\ref{eq4_M_h}), (\ref{eq4_2_3})  and (\ref{eq4_I_N}), we have
   \[
     |g_{N}^3|\le |\partial_3^2h_{N}^3|+|I^3_N|+|\partial_3r_{N}|\le C(N+1)^2K_0^{N+1}\rho^{2(N+1)(1-\epsilon)-\epsilon}, \quad  \textrm{ in }\Gamma_{\delta'}.
   \]
Since $h_{N+1}^3=h_N^3-u_{N+1}^3$ and $\lim_{\rho\to 0^+}h_N^3=\lim_{\rho\to 0^+}u_{N+1}^3=0$, we have $\lim_{\rho\to 0^+}h_{N+1}^3=0$.
 Apply Lemma \ref{lem2_1} with $L=0$ and $\alpha=2(N+1)(1-\epsilon)-\epsilon$ to the third equation of (\ref{eq4_2_1}).
 As argued in step 3,
 $(N+1)/C_{\epsilon}\le |\alpha|+1\le C_{\epsilon}(N+1)$ for some constant $C_{\epsilon}>0$ and all $N\ge 0$. The quantitative estimate (\ref{eq2_1_2}), together with (\ref{eq3_E_1}), gives,
  for $0\le k\le 2$ and $x\in \Gamma_{\delta'}$, that
 \begin{equation}\label{eq4_2_10}
 \begin{split}
 |\nabla^k_{x'}h_{N+1}^3|
    &   \le  C_3(N+1)^kK_0^{N+1}\rho^{2(N+1)(1-\epsilon)-\epsilon+2-k}
       \le C_3(N+1)^kK_0^{N+1}\rho^{2(N+2)(1-\epsilon)-k},
      \end{split}
   \end{equation}
   where $C_3>0$ depends on $(u, p), \delta$ and $\epsilon$.

   \medskip

 Finally, choose
  \[
  K_0>\max\{1, C_3, C_p, C', \bar{K}_0\}(M_0+1),
  \]
   and then choose $0<\delta'<\bar{\delta}$ satisfying (\ref{eq4_2_4}). Combining   (\ref{eq4_2_16}),  (\ref{eq4_2_23}) and (\ref{eq4_2_10}),
  we obtain (\ref{eq4_M_h}) for $N+1$. The proof is finished.
  \end{proof}

  \noindent\emph{Proof of Theorem \ref{thm_main}}:
    The existence of the solutions $(u_{\lambda}, p_{\lambda})$ follows from Lemma \ref{lem4_M_1} (i).

  We next prove the converse part of the theorem.
  Let $\tilde{\delta}>0$ and $(u, p)$ be a $C^2$ $(-1)$-homogeneous axisymmetric solution of (\ref{NS}) in $\Omega_{\tilde{\delta}}\setminus\{x'=0\}$ satisfying $|u|=o(|x'|^{-1})$ as $x'\to 0$ for each fixed $x_3>0$. By Lemma \ref{lem3_1}, there exists
  $\lambda=(c_0, d_0, b_{0}^{\rho}, b_{0}^{\phi})\in \R^4$ associated with $(u, p)$ through (\ref{eq1_a_3}). Let $\{(u_n, p_n)\}_{n=0}^{\infty}$ be the sequence constructed in Lemma \ref{lem3_2} from this $\lambda$, and define $h_N, r_N$ by (\ref{eq4_M_2_N}) for $N\in \mathbb{N}$.
    Applying Lemma \ref{lem4_M_1} (i) with $K=\{\lambda\}$, there exists $\delta_{\lambda}>0$, such that the series $\sum_{n=0}^\infty u_n$ and $\sum_{n=0}^\infty p_n$ converge in $C^2_{loc}(\Omega_{\delta_{\lambda}}\setminus\{x'=0\})$ and define a smooth $(-1)$-homogeneous axisymmetric solution $(u_{\lambda}, p_{\lambda})$ of (\ref{NS}) in $\Omega_{\delta_{\lambda}}\setminus\{x'=0\}$.
    By Lemma \ref{lem4_M_2}, for any fixed $0<\epsilon<1/2$, there exist $0<\delta'<\tilde{\delta}$ and $K_0>0$ such that $h_{N}, r_{N}$ satisfy (\ref{eq4_M_h}) in $\Gamma_{\delta'}$ for $N\ge 0$.

  Choose $0<\delta_1\le \min\{\delta', \delta_{\lambda}, 1\}$ sufficiently small such that
  \[
  K_0\delta_1^{2(1-\epsilon)}<1.
  \]
 Let  $V\subset \subset \Omega_{\delta_1}\setminus\{x'=0\}$ be an arbitrary compact subset.  Then, by (\ref{eq4_M_h}) and the homogeneity of $h_N$ and $r_N$, together with (\ref{eqE_0_4}), we have
\[
\|h_N\|_{C^2(V)}+\|r_N\|_{C^2(V)}\le C_V(N+1)^3K_0^{N+1}\delta_1^{2(N+1)(1-\epsilon)}\to 0,
\quad \textrm{as }N\to \infty,
\]
where $C_V>0$ is a constant depending only on $V$.
 By Lemma \ref{lem4_M_1} we have  $\|\sum_{n=0}^{N}u_n-u_{\lambda}\|_{C^2(V)}\to 0$ as $N\to \infty$. So
   \[
    \|u-u_{\lambda}\|_{C^2(V)}\le  \|\sum_{n=0}^{N}u_n-u_{\lambda}\|_{C^2(V)}+\|h_N\|_{C^2(V)}\to 0,
  \]
as $N\to \infty$. Thus $u=u_{\lambda}$ in $\Omega_{\delta_1}\setminus\{x'=0\}$. Similarly, we have $p=p_{\lambda}$ in $\Omega_{\delta_1}\setminus\{x'=0\}$.

Finally, we prove the uniqueness of the $(-1)$-homogeneous axisymmetric solution $(u_{\lambda}, p_{\lambda})$ of (\ref{NS}) in $\Omega_{\delta}\setminus\{x'=0\}$
 satisfying (\ref{eq1_a_3}) and $|u_{\lambda}|=o(|x'|^{-1})$ as $x'\to 0$ for each fixed $x_3>0$.
Suppose that $(u,p)$ is another such solution
with the same parameter $\lambda$. The argument above shows that $(u,p)=(u_{\lambda}, p_{\lambda})$ in $\Omega_{\delta_1}\setminus\{x'=0\}$ for some $0<\delta_1<\delta$. Since both solutions are $(-1)$-homogeneous and axisymmetric, their traces on $\{x_3=1\}$ solve the same ODE system (\ref{eqNC_C}) for $0<\rho<\delta$.
Choose any $\rho_0\in (0, \delta_1)$. The local equality above implies the two solutions have the same Cauchy data at $\rho_0$. Since every point in $(0, \delta)$ is an ordinary point of (\ref{eqNC_C}), standard ODE uniqueness theory implies that $(u, p)=(u_{\lambda}, p_{\lambda})$ for $0<\rho<\delta$.

  The proof is finished.
  \qed

  \medskip

\noindent  \emph{Proof of Proposition \ref{prop1_1}}: It suffices to prove the result for $L\ge 2$, since the cases $L=0, 1$ follow from the case $L=2$.
By Lemma \ref{lem4_M_1} (ii), repeated use of (\ref{eqE_0_1}), and the relations between Cartesian and cylindrical coordinates, we obtain  (\ref{eqthm1_2}) on $\Gamma_{\delta}$. Since $u_{\lambda}$ is $(-1)$-homogeneous and $p_{\lambda}$ is $(-2)$-homogeneous, we have (\ref{eqthm1_2}) in $\Omega_{\delta}\setminus\{x'=0\}$.
\qed

\section{Singular force of Type II solutions}
 \label{sec_F}

In this section, we compute the force generated across the singular ray $\{x'=0, x_3>0\}$ by the local $(-1)$-homogeneous axisymmetric solutions of (\ref{NS}) constructed in Theorem \ref{thm_main}, and prove Proposition \ref{propF}.

\medskip

\noindent\emph{Proof of Proposition \ref{propF}}:
The proof follows the argument in \cite{LY} for global Type II solutions, which necessarily have no swirl. We include the details needed to treat the local setting and the possible presence of nonzero swirl.

\noindent 1. Let $\delta>0$ and $\Omega=\Omega_{\delta}$ be defined by (\ref{eq_Cone}).  Let $\lambda=(c_0, d_0, b_0^{\rho}, b_0^{\phi})\in \R^4$ and $(u, p)=(u_{\lambda}, p_{\lambda})$ be as in Proposition  \ref{propF}.
For $h_1<h_2$ and $s>0$, define
\[
 V_{h_1, h_2, s, \delta}:=\{x\in \Omega_{\delta}\mid |x'|<s, h_1<x_3<h_2\}.
\]
  Let
\[
   T_{ij}:=p\delta_{ij}+u_iu_j-\partial_ju_i-\partial_iu_j, \quad 1\le i, j \le 3,
   \]
   be the stress tensor.

 Recall that $(r, \phi, x_3)$ are  the standard cylindrical coordinates, $\rho=r/x_3$, $e_{\rho}=e_r$ and $u^{\rho}=u^r$.
 By Theorem \ref{thm_main} and Proposition \ref{prop1_1}, we have, for $x\in \Omega_{\delta}\setminus\{x'=0\}$, that
        \be\label{eqF_1_1}
            |u|\le \frac{C_0(|\ln \rho|+1)}{x_3}, \quad |p|\le \frac{C_0(|\ln \rho|+1)}{x_3^2}, \quad |\nabla u|\le\frac{C_0}{x_3^2\rho}, \quad |T_{ij}|\le \frac{C_0}{x_3^2\rho},
            \ee
            for some constant $C_0>0$ depending only on $\lambda$ and $\delta$.
            Consequently, $T_{ij}\in L^q_{loc}(\Omega)$ for any $1\le q<2$ and $1\le i, j\le 3$.

 Let $\varphi\in C_c^{\infty}(\Omega)$ be a test function. Let $C>0$ denote a constant depending on $(u, p)$, $\varphi$ and $\delta$, which may vary from line to line, and $O(1)$ be a quantity bounded by such a constant.
 Then
 there exist $h_2, R>0$ such that $\supp \varphi\subset\subset V_{0, h_2, R, \delta}$.
  For any $s>0$, denote
 $
   V_s:=V_{0, h_2, s, \delta}.
 $
Then
\[
   \partial V_{s}\cap\{|x'|=s\}=\{|x'|=s,  s/\delta\le x_3\le h_2\}, \quad \forall s>0.
\]
Since $\supp \varphi$ is compactly contained in $\Omega$, there exists some $0<\epsilon_0<\min\{R, 1\}$, such that for all $0<\epsilon<\epsilon_0$, $\supp \varphi\subset \{|x'|<R, \epsilon/\delta<x_3<h_2\}$.

\medskip

\noindent 2. We first identify the divergence-free condition across the singular ray.
Arguing as in \cite{LY}, we claim that
\be\label{eqF_2_0}
  \int_{\Omega}u  \cdot\nabla \varphi dx=0.
\ee
Indeed, using (\ref{eqF_1_1}), the fact that $\dive u=0$ in $V_R\setminus V_{\epsilon}$, and the inclusion $\supp \varphi\subset\subset V_R$, which implies $\varphi=0$ on $\partial (V_R\setminus V_{\epsilon})\setminus (\partial V_{\epsilon}\cap\{|x'|=\epsilon\})$,  we have, for $0<\epsilon<\epsilon_0$, that
\[
\begin{split}
 &  | \int_{\Omega}u\cdot \nabla \varphi dx|   \le |\int_{V_R\setminus V_{\epsilon}}u\cdot \nabla \varphi dx|+|\int_{V_{\epsilon}}u\cdot \nabla \varphi dx|
   \le |\int_{\partial V_{\epsilon}\cap\{|x'|=\epsilon\}}u \cdot\nu \varphi d\sigma| +|\int_{V_{\epsilon}}u\cdot \nabla \varphi dx|\\
   & \le  C\epsilon\int_{\epsilon/\delta}^{h_2}\frac{|\ln(\epsilon/x_3)|+1}{x_3}dx_3+C \int_{\epsilon/\delta}^{h_2}\int_0^{\epsilon}\frac{r(|\ln (r/x_3)|+1)}{x_3}dr dx_3 \le C\epsilon (|\ln\epsilon|^2+1),
   \end{split}
\]
which tends to zero as $\epsilon\to 0$. Here $\nu$ denotes the outward unit normal of $\partial V_{\epsilon}$.

\medskip

\noindent 3. Next, we study the first equation in (\ref{NS}) across $\{x'=0, x_3>0\}$.
Let
\[
   F_j[\varphi]:=\int_{\Omega}T_{ij}\partial_i\varphi dx.
\]
 Below we let $0<\epsilon<\epsilon_0$. By (\ref{NS}), we have $\partial_iT_{ij}=0$ in $\Omega\setminus\{x'=0\}$, $1\le j\le 3$. Integration by parts, using the fact that  $\varphi=0$ on $\partial (V_R\setminus V_{\epsilon})\setminus (\partial V_{\epsilon}\cap\{|x'|=\epsilon\})$,
 we have
\begin{equation*}%\label{eqF_6}
\begin{split}
F_j[\varphi] & =\int_{V_R}T_{ij}\partial_i\varphi dx  =\int_{V_R\setminus V_{\epsilon}}T_{ij}\partial_i\varphi dx+\int_{V_\epsilon}T_{ij}\partial_i\varphi dx=-L_j+\int_{V_\epsilon}T_{ij}\partial_i\varphi dx, 
   %-F_j[\varphi] & =-\int_{V_R}T_{ij}\partial_i\varphi dx  =-\int_{V_R\setminus V_{\epsilon}}T_{ij}\partial_i\varphi dx-\int_{V_\epsilon}T_{ij}\partial_i\varphi dx=L_j-\int_{V_\epsilon}T_{ij}\partial_i\varphi dx %\\
 % &  =\int_{\partial V_{\epsilon}\cap\{|x'|=\epsilon\}}T_{ij} \nu_i \varphi d\sigma-\int_{V_\epsilon}T_{ij}\partial_i\varphi dx.
          \end{split}
\end{equation*}
where
\[
   L_j :=\int_{\partial V_{\epsilon}\cap\{|x'|=\epsilon\}}T_{ij}\nu_i \varphi d\sigma.
\]
Since  $T_{ij}\in L^1_{loc}(\Omega)$, we have $\int_{V_\epsilon}T_{ij}\partial_i\varphi dx=o_{\epsilon}(1)$, and
\begin{equation}\label{eqF_4}
   -F_j[\varphi] %= -\int_{V_R}T_{ij}\partial_i\varphi dx
    =L_j+o_{\epsilon}(1), \quad j=1, 2, 3,
\end{equation}
where $o_{\epsilon}(1)\to 0$ as $\epsilon\to 0$.
Write
\[
  L_j=L_j^{(1)}+L_j^{(2)},
\]
where
\[
  L_j^{(1)}=\int_{\partial V_{\epsilon}\cap\{|x'|=\epsilon\}}T_{ij} \nu_i \varphi(0,0,x_3)d\sigma, \quad L_j^{(2)}=\int_{\partial V_{\epsilon}\cap\{|x'|=\epsilon\}}T_{ij} \nu_i (\varphi-\varphi(0,0,x_3))d\sigma.
\]
By the smoothness of $\varphi$, we have $\varphi(x)=\varphi(0, 0, x_3)+O(|x'|)$ on $\{|x'|=\epsilon\}$. By this, (\ref{eqF_1_1}), the fact that  $\rho=\epsilon/x_3$ on $\partial V_{\epsilon}\cap\{|x'|=\epsilon\}$, we have, for $j=1,2,3$, that
\begin{equation}\label{eqF_5}
    |L_j^{(2)}|\le C\epsilon \sum_{i=1}^{3}\int_{\partial V_{\epsilon}\cap\{|x'|=\epsilon\}}|T_{ij}|d\sigma\le C\epsilon\int_{\epsilon/\delta}^{h_2}\frac{1}{x_3}dx_3\le C\epsilon(|\ln\epsilon|+1)\to 0, \textrm{ as } \epsilon\to 0.
\end{equation}

\noindent Step 3-1. We show
 \be\label{eqF_3_1}
      L^{(1)}_j=0, \quad j=1,2, \quad \forall 0<\epsilon<\epsilon_0.
      \ee
  We only prove (\ref{eqF_3_1}) for $j=1$. The proof for $j=2$ is similar.

   In cylindrical coordinates $(r, \phi, x_3)$, the associated orthonormal basis vectors are
$e_r=e_{\rho}=(\cos\phi,\sin\phi,0)$, $e_{\phi}=(-\sin\phi, \cos\phi, 0)$,   $e_{3}=(0,0,1)$.
We write
  $x=r e_r+x_3e_3$, $x'=r e_r$, $u=u^re_r+u^{\phi}e_{\phi}+u^3e_3$, and $u'=u^re_r+u^{\phi}e_{\phi}$. Since $(u, p)$ is axisymmetric, we have
 $u^r(=u^{\rho}), u^{\phi}, u^3, p$ are independent of $\phi$. By computation,
 \be\label{eqF_3_2}
\begin{split}
    & x'\cdot u'=r u^r,\quad
    \nabla (x'\cdot u')=\nabla(r u^r)=\partial_r (r u^r)e_r+\partial_3 (r u^r)e_3, \quad x'\cdot\nabla' u_i=r\partial_ru_i,
   \end{split}
\ee
where $i=1, 2, 3$.
On $\partial V_{\epsilon}\cap \{|x'|=\epsilon\}$, the outward unit normal vector is $\nu=\frac{1}{r}(x_1,x_2, 0)$. Note
$u_1=u^r\cos\phi-u^{\phi}\sin\phi$. Using the above, we have
\[
  \begin{split}
   & T_{i1} \nu_i  =\frac{1}{r}\left(px_1+x'\cdot u'u_1-x'\cdot\nabla' u_1-\partial_1(x'\cdot u')+u_1\right)\\
                    & =\frac{1}{r}\left(pr\cos\phi+r u^r(u^r\cos\phi-u^{\phi}\sin\phi)-r\partial_r(u^r\cos\phi-u^{\phi}\sin\phi)-\partial_r(ru^r)\cos\phi+u^r\cos\phi-u^{\phi}\sin\phi\right)\\
                   &=G_1(r,x_3)\cos\phi+G_2(r, x_3)\sin\phi,
  \end{split}
\]
where
\[
   G_1(r, x_3)=p+|u^r|^2-2\partial_ru^r, \quad G_2(r, x_3)=\frac{1}{r}(-ru^ru^{\phi}+r\partial_ru^{\phi}-u^{\phi}).
\]
So
\[
\begin{split}
    L_1^{(1)} & = \int_{\partial V_{\epsilon}\cap\{|x'|=\epsilon\}}T_{i1}\nu_i \varphi(0,0,x_3)d\sigma\\
    &  =\epsilon\int_{\epsilon/\delta}^{h_2}\left(G_1(\epsilon,x_3)\int_{0}^{2\pi}\cos\phi d\phi+G_2(\epsilon,x_3)\int_{0}^{2\pi}\sin\phi d\phi\right)\varphi(0,0,x_3)dx_3
     =0.
          \end{split}
\]
By a similar argument we also have $L_2^{(1)}=0$. So (\ref{eqF_3_1}) is proved.

\medskip

\noindent Step 3-2. We show
\be\label{eqF_4_1}
  \lim_{\epsilon \to 0}L^{(1)}_3=2\pi c_0\int_{0}^{h_2}\ln |x_3|\partial_3\varphi(0,0,x_3)dx_3.
\ee
   Recall
 \[
    L_3^{(1)}=\int_{\partial  V_{\epsilon}\cap\{|x'|=\epsilon\}}T_{i3}\nu_i\varphi(0,0,x_3)d\sigma,
 \]
 and
 $
    T_{i3}=u^iu^3-\partial_3u^i-\partial_iu^3
 $, $i=1, 2$.
 Using (\ref{eqF_3_2}), we obtain
 \be\label{eqF_4_2}
   \begin{split}
   & T_{i3} \nu_i =\frac{1}{r}(x'\cdot u'u^3-x'\cdot\nabla' u^3-\partial_3(x'\cdot u'))
    =u^ru^3-\partial_r u^3-\partial_3u^r.
   \end{split}
 \ee
 Note $\rho=r/x_3\le \delta$ in $\Omega$. By the explicit formula (\ref{eq1_a_2}) of the leading terms, the estimates in the Claim of Lemma \ref{lem4_M_1},
  and (\ref{eqE_0_1}), we have
   \[
  u^3=O(1)\frac{1}{x_3}(|\ln \frac{r}{x_3}|+1),\quad  \partial_r u^3
  =\frac{c_0}{x_3r}+O(1)\frac{1}{x_3^2}(|\ln \frac{r}{x_3}|+1),
 \]
 \[
    u^r=O(1)\frac{r}{x_3^2}(|\ln \frac{r}{x_3}|+1), \quad \partial_3u^r
    =O(1)\frac{r}{x_3^3}(|\ln\frac{r}{x_3}|+1).
 \]
By (\ref{eqF_4_2}) and the above, we have
  \be\label{eqF_4_3}
   T_{i3} \nu_i=-\frac{c_0}{x_3r}+O(1)\frac{1}{x_3^2}(|\ln \frac{r}{x_3}|+1)=: -\frac{c_0}{x_3r}+G(x).
 \ee
  Then
 \[
 \begin{split}
   L_3^{(1)} & =\int_{\partial V_{\epsilon}\cap\{|x'|=\epsilon\}}T_{i3}\nu_i\varphi(0,0,x_3)d\sigma\\
   & =-c_0\int_{\partial V_{\epsilon}\cap\{|x'|=\epsilon\}}\frac{\varphi(0,0,x_3)}{x_3r}d\sigma+\int_{\partial V_{\epsilon}\cap\{|x'|=\epsilon\}}G(x)\varphi(0,0,x_3)d\sigma
    =:A+B.
   \end{split}
 \]
Since $\supp \varphi\subset \{|x'|<R, \epsilon/\delta<x_3<h_2\}$, we have $\varphi(0, 0, h_2)=\varphi(0, 0, \epsilon/\delta)=0$
and $\varphi(0, 0, x_3)=O(x_3)$.
Then
\[
   \begin{split}
     A &
     =-2\pi c_0\int_{\epsilon/\delta}^{h_2}\frac{\varphi(0, 0, x_3)}{x_3}dx_3\\
     &  =-2\pi c_0\ln |x_3|\varphi(0, 0, x_3)|_{\epsilon/\delta}^{h_2}+2\pi c_0\int_{\epsilon/\delta}^{h_2}\ln |x_3| \partial_3\varphi(0, 0, x_3)dx_3\\
     & =
     2\pi c_0\int_{\epsilon/\delta}^{h_2}\ln |x_3| \partial_3\varphi(0, 0, x_3)dx_3.
   \end{split}
\]
On the other hand, by (\ref{eqF_4_3}) and the fact $\varphi(0,0,x_3)=O(x_3)$, we have
\[
  |B|\le C\epsilon\int_{\epsilon/\delta}^{h_2}\frac{1}{x_3^2}(|\ln \frac{\epsilon}{x_3}|+1)\cdot x_3dx_3\le C\epsilon (|\ln\epsilon|^2+1).
\]
Thus
\[
  L_3^{(1)}=A+B=O(\epsilon(|\ln\epsilon|^2+1))+2\pi c_0\int_{\epsilon/\delta}^{h_2}\ln |x_3| \partial_3\varphi(0, 0, x_3)dx_3.
\]
Since $\ln |x_3| \partial_3\varphi(0, 0, x_3)\in L^1(0, h_2)$, sending $\epsilon\to 0$ in the above, we have  (\ref{eqF_4_1}).

By (\ref{eqF_2_0}), we have $\dive u=0$ in the sense of distributions. Moreover, applying (\ref{eqF_2_0}) with $\partial_j\varphi$ in place of $\varphi$, we obtain
\[
   \int_{\Omega}\partial_ju^i\partial_i\varphi dx=-\int_{\Omega} u^i\partial_i\partial_j\varphi dx=0.
\]
It follows from the definition of $T_{ij}$ that
\[
   -F_j[\varphi]=\int_{\Omega}(\nabla u^j\cdot \nabla \varphi-u^iu^j\partial_i\varphi-p\partial_j\varphi) dx.
\]
Combining this with %(\ref{eqF_6}),
 (\ref{eqF_4}), (\ref{eqF_5}), (\ref{eqF_3_1}) and (\ref{eqF_4_1}), we have (\ref{eqF_2}). Proposition \ref{propF} is proved.
\qed

\end{document}